\documentclass[10.5pt,reqno]{amsart}
\usepackage{amsfonts}
\usepackage{indentfirst}
\usepackage{graphicx}
\usepackage{amssymb}
\usepackage{color,xcolor}
\usepackage{subfig}
\usepackage{graphicx,epstopdf}
\usepackage{caption}
\usepackage{mathrsfs}
\usepackage{bbm}
\usepackage{chngpage}
\usepackage{cite}
\usepackage{multirow}
\usepackage{multicol}
\usepackage{dsfont}
\usepackage{bm}
\usepackage{amssymb,amsmath,amsthm,mathrsfs}%
\usepackage{exscale}
\usepackage{tikz-cd}
\usepackage{relsize}
\usepackage{amssymb}
\usepackage{hyperref}
\usepackage{url}
\usepackage{tikz-cd}
\usepackage{comment}
\usepackage[misc,geometry]{ifsym} 
\usepackage{pst-solides3d}

\usepackage{enumitem}
\allowdisplaybreaks

\hypersetup{hypertex=green,
            colorlinks=true,
            linkcolor=green,
            anchorcolor=green,
            citecolor=red}
\theoremstyle{definition}%{remark}{\plain}
\newtheorem{theorem}{Theorem}[section]

\newtheorem{example}{Example}

\newtheorem{remark}{Remark}[section]
\newtheorem{lemma}{Lemma}[section]

\newtheorem{proposition}{Proposition}[section]

\numberwithin{equation}{section}%
\numberwithin{table}{section}%
\numberwithin{figure}{section}

\def\3bar{{|\hspace{-.02in}|\hspace{-.02in}|}}
\newcommand\tr{\operatorname{tr}}

\newcommand\grad{\operatorname{grad}}
\renewcommand\div{\operatorname{div}}

\newcommand\curl{\operatorname{curl}}

\def\d{\text{d}}

\begin{document}
\title[Nonconforming finite elements]{Nonconforming finite elements for the $-\curl\Delta\curl$ and Brinkman problems on cubical meshes}

\author{Qian Zhang}
\email{qzhang15@mtu.edu}
\address{Department of Mathematical Sciences, Michigan Technological University, Houghton, MI 49931, USA. }

\author{Min Zhang}
\email{zhangminyy@csrc.ac.cn (corresponding author)}
\address{Beijing Computational Science Research Center, Beijing, China}

\author{Zhimin Zhang}
\email{zmzhang@csrc.ac.cn; zzhang@math.wayne.edu}
\address{Beijing Computational Science Research Center, Beijing, China; Department of Mathematics, Wayne State University, Detroit, MI 48202, USA}
\thanks{This work is supported in part by the National Natural Science Foundation of China grants NSFC 12131005 and NSAF U2230402.}


%\date{\today}
\begin{abstract} 
We propose two families of nonconforming elements on cubical meshes: one for the $-\curl\Delta\curl$ problem and the other for the Brinkman problem. The element for the $-\curl\Delta\curl$ problem is the first nonconforming element on cubical meshes. The element for the Brinkman problem can yield a uniformly stable finite element method with respect to the viscosity coefficient $\nu$. The lowest-order elements for the $-\curl\Delta\curl$ and the Brinkman problems have 48 and 30 DOFs on each cube, respectively. 
 The two families of elements are subspaces of $H(\curl;\Omega)$ and $H(\div;\Omega)$, and they, as nonconforming approximation to  $H(\grad\curl;\Omega)$ and $[H^1(\Omega)]^3$, can form a discrete Stokes complex together with the serendipity finite element space and the piecewise polynomial space.
\end{abstract}	
\maketitle

\section{Introduction}
Let $\Omega\subset \mathbb R^3$ be a contractible Lipschitz polyhedral domain. For $\bm  f\in H(\div^0;\Omega)$, we consider the following $-\curl\Delta\curl$ problem:
\begin{equation}\label{prob1}
\begin{split}
-\mu\curl\Delta \curl\bm u+\curl\curl\bm u+\gamma\bm u&=\bm f\ \ \text{in}\;\Omega,\\
\div \bm u &= 0\ \ \  \text{in}\;\Omega,\\
\bm u\times\bm n&=0\ \ \text{on}\;\partial \Omega,\\
\curl \bm u&=0\ \  \text{on}\;\partial \Omega.
\end{split}
\end{equation}
Here $\gamma\geq 0$ is a constant of moderate size, $\mu>0$ is a constant that can approach 0, $\bm n$ is the unit outward normal vector on $\partial \Omega$, and
 $H(\div^0;\Omega)$ is the space of $[L^2(\Omega)]^3$ functions with vanishing divergence, i.e., 
\[H(\text{div}^0;\Omega) :=\{\bm u\in [L^2(\Omega)]^3:\; \div \bm u=0\}.\]
Problem \eqref{prob1} arises in applications related to electromagnetism and continuum mechanics \cite{mindlin1962effects, park2008variational, chacon2007steady}.
 Conforming finite element approximations of this problem require the construction of finite element spaces that belong to $H(\grad\curl;\Omega)$, which are commonly referred to as $H(\grad\curl)$-conforming finite element spaces. Recently, two of the authors, along with their collaborators, developed three families of $H(\grad\curl)$-conforming elements on both triangular and rectangular meshes \cite{Hu2020Simple, Hu2020Simple2, WZZelement}. The corresponding spectral construction of the three families of rectangular elements is detailed in \cite{wang2021h}. In three-dimensional space, two of the authors proposed a tetrahedral $H(\grad\curl)$-conforming element \cite{ZhangCSIAM2021A} consisting of 315 DOFs per element, which was later improved by enriching the shape function space with piecewise-polynomial bubbles to reduce the DOFs to 18 \cite{Hu2022A}. While the construction of $H(\grad\curl)$-conforming elements in two dimensions and on tetrahedral meshes is relatively complete, the development of cubical elements remains a challenge. The only cubical $H(\grad\curl)$-conforming element in the literature has 144 DOFs \cite{wang2021hh}. To address the issue of high DOFs in cubical elements, nonconforming finite elements may offer a viable solution. The existing literature reports two low-order nonconforming elements on tetrahedral meshes \cite{Zheng2011A, huang2020nonconforming}.
However, as far as we are aware, there has been no previous research on the construction of nonconforming cubical elements, which is one objective of this paper. In contrast to the aforementioned nonconforming tetrahedral elements, which have low order accuracy and limited potential for extension to higher orders, our proposed cubical elements are capable of achieving arbitrary order accuracy.

A related problem is the Brinkman model of porous flow, which seeks $(\bm u;p)$ such that 
\begin{equation}\label{prob-brinkman}
\begin{aligned}
-\div(\nu \grad\bm u)+\alpha\bm u+\grad p&=\bm f\ \ \text{in}\;\Omega,\\
\div \bm u &=g\ \ \text{in}\;\Omega,\\
\bm u&=0\ \  \text{on}\;\partial \Omega.
\end{aligned}
\end{equation}
Here $\bm u$ is the velocity, $p$ is the pressure, $\alpha>0$ is the dynamic viscosity divided by the permeability, $\nu>0$ is the effective viscosity, and $\bm f\in [L^2(\Omega)]^3$ and $g\in L^2(\Omega)$  are two forcing terms. We assume $\alpha$ is a moderate constant, $\nu$ is a constant that can approach 0, and $g$ satisfies the compatibility criterion $\int_{\Omega}g\d V=0.$
The Brinkman problem is used to describe the flow of viscous fluids in porous media with fractures. Applications of this model include the petroleum industry, the automotive industry, underground water hydrology, and heat pipes modeling. Depending on the value of effective viscosity $\nu$, the Brinkman problem can be locally viewed as a Darcy or Stokes problem.
When the Brinkman problem is Darcy-dominating ($\nu$ tends to 0),
applying stable Stokes finite element pairs such as the Crouzeix-Raviart  element \cite{crouzeix1973conforming}, the Mini element \cite{arnold1984stable}, and the Taylor-Hood elements \cite{taylor1973numerical} will lead to non-convergent discretizations. Similarly, when the Brinkman problem is Stokes-dominating ($\nu$ is a moderate size), stable  Darcy finite elements such as Raviart-Thomas (RT) elements will fail.
Since in real applications the number and location of Stokes-Darcy interfaces are usually unknown, it is desirable to construct finite elements that are uniformly stable with respect to the parameter $\nu$. It is shown in \cite{xie2008uniformly} that a pointwise divergence-free finite element pair yields a uniformly stable numerical scheme. Conforming finite elements that satisfy this requirement include high-order Scott-Vogelius triangular elements on singular-vertex-free meshes \cite{guzman2019scott}, some macro finite elements \cite{arnold1992quadratic,xu2010new,zhang2009family,fu2020exact,guzman2020exact,christiansen2018generalized,zhang2008p1,zhang2005new,zhang2011quadratic}, and the elements in \cite{guzman2018inf,neilan2015discrete,Hu2022A,neilan2016stokes}. Regarding nonconforming elements,  Tai et al constructed low-order $[H^1]^3$-nonconforming but $H(\div)$-conforming elements on triangular \cite{mardal2002robust} and tetrahedral meshes \cite{tai2006discrete} by modifying  existing $H(\div)$-conforming elements. Later in 2012, Guzm\'an and Neilan \cite{johnny2012family} proposed a family of nonconforming elements on simplicial meshes in both 2D and 3D by using an idea similar to Tai et al's.  In 2017, Chen et al. further extended this idea to cubical meshes, constructing a low-order element with 33 DOFs \cite{chen2017uniformly}. However, the general construction of arbitrary order elements is not yet available in the literature, which is the second objective of this paper.  It is worth noting that there are other uniformly stable nonconforming elements that are not $H(\div)$-conforming, such as those mentioned in \cite{zhou2021low}.

The two model problems are closely related by the following Stokes complex:
\begin{equation}\label{stokes-complex}
\begin{tikzcd}
0 \arrow{r} & \mathbb R\arrow{r}{\subset}& H^{1}(\Omega) \arrow{r}{\grad} &H(\grad\curl;\Omega)\arrow{r}{\curl} & {[H^1(\Omega)]}^3   \arrow{r}{\div} &L^2(\Omega)  \arrow{r}{} & 0.
 \end{tikzcd}
\end{equation}
In this paper, we construct a nonconforming finite element Stokes complex for $r\geq 2$:
\begin{equation}\label{nonconforming-FE-derham-complex}
\begin{tikzcd}
0 \arrow{r} & \mathbb R\arrow{r}{\subset}& \mathcal S_r^0(\mathcal T_h) \arrow{r}{\grad} & \mathcal S_{r-1,r}^{+,1}(\mathcal T_h)\arrow{r}{\curl_h} &  \mathcal S_{r-1}^{+,2}(\mathcal T_h) \arrow{r}{\div_h} & \mathcal S_{r-2}^3(\mathcal T_h) \arrow{r}{} & 0
 \end{tikzcd}
\end{equation} 
on cubical meshes, and prove its exactness on contractible domains.
The main challenge lies in the construction of the $H(\operatorname{grad}\operatorname{curl})$-nonconforming finite element spaces $\mathcal S_{r-1,r}^{+,1}(\mathcal T_h)$ and the $[H^1]^3$-nonconforming finite element spaces $\mathcal S_{r-1}^{+,2}(\mathcal T_h)$. 
The works \cite{johnny2012family,xie2008uniformly,tai2006discrete} provide uniformly-stable nonconforming finite element spaces for Darcy-Stokes-Brinkman models by enriching the $H(\div)$-conforming finite element with some  bubble functions to enforce extra smoothness. This motivates us to construct $\mathcal S_{r-1,r}^{+,1}(\mathcal T_h)$ and $\mathcal S_{r-1}^{+,2}(\mathcal T_h)$ by addressing two key issues: firstly, determining two conforming finite element spaces $\mathcal S_{r-1,r}^{1}(\mathcal T_h)\subseteq H(\operatorname{curl};\Omega)$ and $\mathcal S_{r-1}^{2}(\mathcal T_h)\subseteq H(\operatorname{div};\Omega)$; and secondly, carefully selecting bubble function enrichments to enforce extra smoothness so that consistency errors can be bounded.

To this end, we first construct a discrete sub-complex of the de Rham complex for $r\geq 2$:
\begin{equation}\label{FE-derham-complex}
\begin{tikzcd}
0 \arrow{r} & \mathbb R\arrow[r,"\subset"]& \mathcal S_r^0(\mathcal T_h) \arrow[d,"\cap"]\arrow[r,"\grad"] & \mathcal S_{r-1,r}^1(\mathcal T_h)\arrow[r,"\curl"]\arrow[d,"\cap"] &  \mathcal S_{r-1}^2(\mathcal T_h)\arrow[r,"\div"]\arrow[d,"\cap"] & \mathcal S_{r-2}^3(\mathcal T_h) \arrow[r]\arrow[d,"\cap"] & 0\\
0 \arrow{r} & \mathbb R\arrow{r}{\subset}& H^{1}(\Omega) \arrow{r}{\grad} &H(\curl;\Omega)\arrow{r}{\curl} & H(\div;\Omega)   \arrow{r}{\div} &L^2(\Omega)  \arrow{r}{} & 0,
 \end{tikzcd}
\end{equation}
and establish the exactness of this complex on a contractible domain. The space $\mathcal S_r^0(\mathcal T_h)\subseteq H^1(\Omega)$ is the serendipity finite element space introduced in \cite[(2.2)]{arnold2011serendipity},
while $\mathcal S_{r-2}^3(\mathcal T_h)$ simply denotes the space of discontinuous piecewise polynomials with degree at most $r-2$. The space $\mathcal S_{r-1,r}^{1}(\mathcal T_h)$ is the trimmed $H(\curl)$-conforming finite element space $\mathcal S_{r}^-\Lambda^1(\mathbb{R}^3)$ constructed in \cite{gillette2019trimmed}, and the $\mathcal S_{r-1}^2(\mathcal T_h)$ is the $H(\div)$-conforming finite element space  $\mathcal S_{r-1}\Lambda^2(\mathbb{R}^3)$ in \cite{arnold2014finite}. In fact, the complex \eqref{FE-derham-complex} is constructed by combining the two complexes in \cite{arnold2014finite} and \cite{gillette2019trimmed} with the specific purpose of ensuring that  $\mathcal S_{r-1,r}^{1}(\mathcal T_h)$ and $\mathcal S_{r-1}^2(\mathcal T_h)$ in the resulting complex \eqref{FE-derham-complex} contain at least piecewise linear polynomials with a minimal number of DOFs.
Then we enrich $\mathcal S_{r-1,r}^{1}(\mathcal T_h)$ and $\mathcal S_{r-1}^{2}(\mathcal T_h)$ with some bubble functions to construct $\mathcal S_{r-1,r}^{+,1}(\mathcal T_h)$ and $\mathcal S_{r-1}^{+,2}(\mathcal T_h)$.
Thanks to the exactness of \eqref{FE-derham-complex}, once the bubble function enrichment for one of the spaces $\mathcal S_{r-1,r}^{1}(\mathcal T_h)$ and $\mathcal S_{r-1}^{2}(\mathcal T_h)$ is determined, it will provide a natural option for the bubble function enrichment for the other space. In our case, 
to construct $\mathcal S_{r-1,r}^{+,1}(\mathcal T_h)$, we enrich $\mathcal S_{r-1,r}^{1}(\mathcal T_h)$ with a bubble function space $\mathcal V^{r-2}(K)$ on each cube $K\in \mathcal T_h$. Then the space $\operatorname{curl} \mathcal V^{r-2}(K)$ produces a bubble function enrichment for $\mathcal S_{r-1}^2(\mathcal T_h)$ to construct $\mathcal S_{r-1}^{+,2}(\mathcal T_h)$.
In the lowest-order case ($r=2$), the $\mathcal S_{r-1,r}^{+,1}(\mathcal T_h)$ and $\mathcal S_{r-1}^{+,2}(\mathcal T_h)$ have 48 and 30 DOFs on each element, respectively. We note that $\mathcal S_{1}^{+,2}(\mathcal T_h)$ has the same convergence property as the one in \cite{chen2017uniformly}, but our element has 3 fewer DOFs on each cube. The reason lies in that $\div\mathcal S_{r-1}^{+,2}(\mathcal T_h)$ contains piecewise constants instead of piecewise linear polynomials.
We also note that our $H(\grad\curl)$-nonconforming elements are $H(\curl)$-conforming, while the elements in \cite{Zheng2011A,huang2020nonconforming} are not.  

We demonstrate that the spaces $\mathcal S_{r}^0(\mathcal T_h)$, $\mathcal S_{r-1,r}^{+,1}(\mathcal T_h)$, $\mathcal S_{r-1}^{+,2}(\mathcal T_h)$, and $\mathcal S_{r-2}^{3}(\mathcal T_h)$ form a nonconforming finite element Stokes complex \eqref{nonconforming-FE-derham-complex} and establish its exactness over contractible domains. At the same time, the complex \eqref{nonconforming-FE-derham-complex} is a conforming finite element de Rham complex with enhanced regularity as the bubble function enrichment brings additional smoothness. 
Furthermore, the element pairs $\mathcal S_{r}^0(\mathcal T_h)\times \mathcal S_{r-1,r}^{+,1}(\mathcal T_h)$ and $\mathcal S_{r-1}^{+,2}(\mathcal T_h)\times \mathcal S_{r-2}^{3}(\mathcal T_h)$ are respectively utilized to the $-\operatorname{curl}\Delta\operatorname{curl}$ problem and the Brinkman model with the stability following directly from the exactness of the complex \eqref{nonconforming-FE-derham-complex}. We prove that the two families of finite element pairs lead to convergent  schemes for the two model problems.

% A main difference between our complex and those in \cite{mardal2002robust,tai2006discrete,johnny2012family,chen2017uniformly} is that our complex starts from an $H^1$-conforming finite element space instead of an  $H^2$-nonconforming finite element space.

The remaining part of the paper is organized as follows. In Section 2, we introduce some notation that we use throughout the paper, and we present some polynomial spaces that will be used in Section 3 to define the discrete de Rham complex 
\eqref{FE-derham-complex}, and consequently, the nonconforming Stokes complex \eqref{nonconforming-FE-derham-complex}.
Section 3 is the main part of the paper. We first construct the nonconforming finite element spaces $\mathcal S_{r-1,r}^{+,1}(\mathcal T_h)$ and $\mathcal S_{r-1}^{+,2}(\mathcal T_h)$. Then we fit them into the nonconforming Stokes complex \eqref{nonconforming-FE-derham-complex}, and establish its exactness over contractible domains. In Sections 4 and 5, we apply the two families of newly proposed finite element spaces  $\mathcal S_{r-1,r}^{+,1}(\mathcal T_h)$ and $\mathcal S_{r-1}^{+,2}(\mathcal T_h)$ to solve the two model problems and provide the convergence analysis.  Numerical experiments are presented to validate the nonconforming elements in Section 6. 

\iffalse
The remaining part of the paper is organized as follows. In Section 2, we introduce some notation that we use throughout the paper, and we present some polynomial spaces \textcolor{cyan}{that will be} used \textcolor{cyan}{in Section 3} to define the discrete de Rham complex 
\eqref{FE-derham-complex} \textcolor{cyan}{and consequent the nonconforming Stokes complex \eqref{nonconforming-FE-derham-complex}, which is also a conforming de Rham complex with extra smoothness.}
Section 3 is the main part of the paper, where we \textcolor{cyan}{first} construct the nonconforming finite element spaces $\mathcal S_{r-1,r}^{+,1}(\mathcal T_h)$ and $\mathcal S_{r-1}^{+,2}(\mathcal T_h)$, and show that the two spaces form an exact discrete complex. In Sections 4 and 5, we apply the two newly proposed finite element spaces to solve the two model problems and provide the convergence analysis.  Numerical experiments are presented to validate the nonconforming elements in Section 6. 
%Finally, we summarize our results and give possible extensions in Section 7.
% In Section 3, we introduce the two model problems and their mixed finite element approximations. We also give the conditions that the finite element spaces should satisfy.   
\fi
\section{Preliminary}
\subsection{Notation}
We assume that $\Omega\subset\mathbb{R}^3$ is a contractible Lipschitz polyhedral domain throughout the paper.  We adopt conventional notations for Sobolev spaces such as $H^m(D)$ or $H^m_0(D)$ on a sub-domain $D\subset\Omega$ furnished with the norm $\left\|\cdot\right\|_{m,D}$ and the semi-norm $\left|\cdot\right|_{m,D}$. The space $L^2(D)$ is equipped with the inner product $(\cdot,\cdot)_D$ and the norm $\left\|\cdot\right\|_D$. 
When $D=\Omega$, we drop the subscript $D$. 
%We use $\mathring L^2(D)$ to denote $L^2$ functions with vanishing mean:
%\[\mathring L^2(D)=\left\{q \in L^2(D): \int_D q \d V=0\right\}.\] We also use  $\bm H^{m}(D)$, $\bm H_0^{m}(D)$,  and ${\bm L}^2(D)$ to denote the vector-valued Sobolev spaces  $[H^{m}(D)]^3$, $[H_0^{m}(D)]^3$, and $[L^2(D)]^3$. 

%Let ${\bm u}=(u_1, u_2,u_3)^T$ and ${\bm w}=(w_1, w_2, w_3)^T$, where the superscript $T$ denotes the transpose,
%then ${\bm u} \times {\bm w} = (u_2w_3-w_2u_3, w_1u_3-u_1w_3, u_1w_2-w_1u_2)^T$ and 
%$\nabla \times {\bm u} = (\partial_{x_2}u_{3}- \partial_{x_3}u_{2}, \partial_{x_3}u_{1} -  \partial_{x_1}u_{3}, \partial_{x_1}u_{2} - \partial_{x_2}u_{1} )^T$.
%For convenience, here and hereinafter we abbreviate the partial differential operators $\frac{\partial }{\partial x_i}$
%to $\partial_{ x_i}$.
%We denote $(\nabla\times)^2\bm u=\nabla\times\nabla\times\bm u$.

In addition to the standard Sobolev spaces, we also define
\begin{align*}
  	& H(\text{curl};D):=\{\bm u \in [L^2(D)]^3:\; \curl \bm u \in [L^2(D)]^3\},\\
  	 &H(\text{div};D):=\{\bm u \in [L^2(D)]^3:\; \div\bm u \in  L^2(D)\},\\
	& H(\grad\curl;D):=\{\bm u \in H(\operatorname{curl};\Omega):\;  \operatorname{curl}  \bm u \in [H^1(D)]^3\},\\
	&H^s(\curl;D):=\{\bm u \in [H^s(D)]^3:\; \curl\bm u \in [H^s(D)]^3\}.
\end{align*}

%If we construct a conforming finite element space of $H(\text{curl}^2;\Omega)$, then the space will be in $H(\grad\curl;\Omega)$, which is the reason that we refer to the elements as curl-curl or grad-curl conforming elements.

%The space of $\bm L^2(D)$ functions with square-integrable divergence is denoted by $H(\td;D)$ and defined by
%\[H(\text{div};D) :=\{\bm u\in {\bm L}^2(D):\; \nabla\cdot \bm u\in L^2(D)\}\]
%with the associated inner product
%$(\bm u,\bm v)_{H(\td;D)}=(\bm u,\bm v)+(\nabla\cdot \bm u,\nabla\cdot \bm v)$ and the norm
%$\left\|\bm u\right\|_{H(\td;D)}=\sqrt{(\bm u,\bm u)_{H(\td;D)}}$.

%We use $Q_{i,j,k}$ to represent the polynomials in three variables $(x_1,x_2,x_3)$ whose maximum degree  are respectively $i$ in $x_1$, $j$ in $x_2$, and $k$ in $x_3$. For simplicity, we drop   subscripts $i$ and $j$ when $i=j=k$. 
For a subdomain $D$, we use $P_r(D)$, or simply $P_{r}$ when there is no possible confusion, to denote the space of polynomials with degree at most $r$ on $D$. We denote $\widetilde{P}_r$ to be the space of homogeneous polynomials of degree $r$.
 %We also define
%\begin{align*}
%&\mathcal R_k=\bm P_{k-1}\oplus \mathcal S_k \text{\ with\ }\mathcal S_k=\{{\bm p}\in \widetilde{\bm P_k}\big| \ \bm x\cdot \bm p=0\},
%%\mathcal D_k=\bm P_{k-1}\oplus  \bm x& \widetilde P_{k-1} \text{\ with\ } \widetilde P_{k-1}\text{\ is the space of homogeneous polynomials of degree } k-1 .
%\end{align*}
%%The dimensions of these special spaces of vector polynomials  are 
%whose dimension is
%\begin{align*}
%&\dim{\mathcal R_k}=\frac{k(k+2)(k+3)}{2}.
%&\dim{\mathcal D_k}=\frac{k(k+1)(k+3)}{2}.
%\end{align*}
%For the space $\bm P_k$, we have the following decomposition \cite{da2018lowest}
%\begin{align}
%&\bm P_k=\nabla P_{k+1} \oplus \bm x\times \bm P_{k-1},\label{Pkdecomp1}\\
%&\bm P_k=\nabla\times \mathcal R_{k+1} \oplus  \bm x P_{k-1}\label{Pkdecomp2}.
%\end{align}
%The dimension of $ \bm x\times \bm P_{k-1}$ is $\dim{\bm P_k}-\dim{P_{k+1}}+1$ and the dimension of $\nabla\times \mathcal R_{k+1}$ is $\dim{\bm P_k}-\dim{P_{k-1}}$.

%We use $Q_{i,j,k}(D)$ to denote the polynomials with three variables $(x_1, x_2, x_3)$ where the maximal degree is $i$ in $x_1$, $j$ in $x_2$, and $k$ in $x_3$. For simplicity, we drop the subscripts $i$ and $j$ when $i=j=k$.  Similarly, we use $Q_{i,j}(f)$ to denote such polynomial spaces in 2D. 

%\cite{piolamapping} 

Let \,$\mathcal{T}_h\,$ be a partition of the domain $\Omega$
consisting of shape-regular cubicals. For $K=(x_1^c-h_1,x_1^c+h_1)\times (x_2^c-h_2,x_2^c+h_2)\times (x_3^c-h_3,x_3^c+h_3)$, we denote $h_K=\sqrt{h_1^2+h_2^2+h_3^2}$ as the diameter of $K$ and $h=\max_{K\in\mathcal T_h}h_K$ as the mesh size of $\mathcal {T}_h$. Denote by $\mathcal V_h$, $\mathcal E_h$, and $\mathcal F_h$ the sets of vertices, edges, and faces in the partition, and $\mathcal V_h(K)$, $\mathcal E_h(K)$, and $\mathcal F_h(K)$ the sets of vertices, edges, and faces related to $K$. 
Denote by $\mathcal F_h^0$ the interior faces in the partition. 
We use $\bm\tau_e$ and $\bm n_f$ to denote the unit tangential vector and the unit normal vector to $e\in \mathcal E_h$ and $f\in \mathcal F_h$, respectively. When it comes to each cube $K\in \mathcal T_h$, we denote $\bm n_{f, K}$ as the unit outward normal vector of $f\in \mathcal F_h(K)$. 
Suppose $f=K_1\cap K_2\in \mathcal F_h^0$. For a function $v$ defined on $K_1\cup K_2$, we define
 $[\![v]\!]_f=f|_{K_1}-f|_{K_2}$ to be the jump across $f$. When $f\subset \partial \Omega$, $[\![v]\!]_f=v$.

We use $C$ to denote a generic positive constant that is independent of $h$.

%the orthogonal projection  $\mathcal{P}_{f}^{r-3}: L^{2}(f)\rightarrow [P_{r-3}(f)]^{2}$ is defined by 
%\begin{equation}
%\langle\mathcal{P}_{f}^{r}\bm{v}, \bm{q}\rangle_{f}= \langle\bm{v}, \bm{q}\rangle_{f}, \quad \text{for all}~~\bm{q}\in [P_{r}(f)], f\in \mathcal{F}_{h}.
%\end{equation}
\subsection{Some polynomial spaces} 
%A polynomial de Rham complex is established and proved to be exact on a contractible domain in this subsection. The proof is easy, while the result plays an important role in the construction of the nonconforming finite element Stokes complexes in the subsequent section.
 
We introduce some polynomial spaces on $K\in\mathcal T_h$. They are frequently used throughout the paper. 
\paragraph{{\bf Space $\mathcal S_r^0(K)$}}
The space $\mathcal S_r^0(K)$ is the serendipity space defined in \cite[Definition 2.1]{arnold2011serendipity}.
For $r\geq 1$, the space $\mathcal S_r^0(K)$ on $K$ contains all polynomials of superlinear degree at most $r$, where the superlinear degree of a polynomial is the degree with ignoring variables which enter linearly, for example, the superlinear degree of $x^2yz^3$ is $5$. This implies $P_{r}(K)\subseteq \mathcal S_r^0(K)\subseteq P_{r+2}(K)$. To be specific, the space $\mathcal S^0_2(K)$ is spanned by the monomials in $P_2(K)$ and  $x_1x_2^2,x_1x_3^2,x_2x_1^2,x_2x_3^2,x_3x_1^2,x_3x_2^2,x_1x_2x_3,x_1x_2x_3^2,x_1^2x_2x_3,x_1x_2^2x_3$. The dimension of the space reads\cite[(2.1)]{arnold2011serendipity}:
\begin{equation}\label{dim:S0:K}
    \dim \mathcal S_1^0(K)=8, \, \dim \mathcal S_2^0(K)=20, \, \text{and}~~\dim \mathcal S_r^0(K)=\frac{1}{6}(r+1)(r^2+5r+24)~~ \text{for}~~r\geq 3.
\end{equation}

\paragraph{{\bf Space $\mathcal S_{r-1,r}^1(K)$}} 
For $r\geq 2$, define 
\begin{equation}\label{s1:onK}
    \mathcal S_{r-1,r}^1(K)=\grad \mathcal S_r^0(K)+[P_{r-1}(K)]^3\times \bm x+\big(x_2x_3(w_2-w_3), x_1x_3(w_3-w_1),x_1x_2(w_1-w_2)\big)^\mathrm T,
\end{equation}
where $w_i,\ i=1,2,3$, belong to $\widetilde{P}_{r-1}(K)$ independent of $x_i$. The dimension of this space is
\begin{equation}\label{dim:S1:K}
\dim{\mathcal S_{1,2}^1(K)}=36\text{ and }\dim{\mathcal S_{r-1,r}^1(K)}=\frac{1}{2}r^3 + \frac{5}{2}r^2+ 9r + 3,\ r\geq 3.
\end{equation}
% In fact, the space $\mathcal S_{r-1,r}^1(K)$ is equivalent to the space $\mathcal S_{r}^-\Lambda^1(\mathbb{R}^3)$ given by \cite[(3.1)]{gillette2019trimmed}, which was used to construct the $H(\curl)$-conforming finite element on cubical grids therein.
\paragraph{{\bf Space $\mathcal S_{r-1}^2(K)$}}
For $r\geq 2$, define
\begin{equation}\label{def:s2}
    \mathcal S_{r-1}^2(K)=[P_{r-1}(K)]^3+\curl\big(x_2x_3(w_2-w_3),x_1x_3(w_3-w_1),x_1x_2(w_1-w_2)\big)^{\mathrm T},
\end{equation}
where $w_i,\ i=1,2,3$, belong to $\widetilde{P}_{r-1}(K)$ independent of $x_i$. The dimension of $\mathcal S_{r-1}^2(K)$ is
\begin{equation}\label{dim:S2:K}
   \dim \mathcal S_{r-1}^2(K)= \frac{1}{2}r(r^2 + 3r + 8).
\end{equation}
% It should be noted that $\mathcal S_{r-1}^2(K)$ is equivalent to the space $\mathcal S_{r-1}\Lambda^2(\mathbb{R}^3)$ defined in \cite[(17)]{arnold2014finite}, which is the shape function space of the $H(\div)$-conforming finite element space on cubical grids therein.
In fact, the polynomial spaces $\mathcal S_r^0(K)$, $\mathcal S_{r-1,r}^1(K)$, and $\mathcal S_{r-2}^2(K)$ are the shape function spaces of $\mathcal S_r^0(\mathcal T_h)$, $\mathcal S_{r-1,r}^1(\mathcal T_h)$, and $\mathcal S_{r-2}^2(\mathcal T_h)$ in complex \eqref{FE-derham-complex}, respectively.
To avoid repetition, the DOFs for defining $\mathcal S_r^0(\mathcal T_h)$, $\mathcal S_{r-1,r}^1(\mathcal T_h)$, and $\mathcal S_{r-2}^2(\mathcal T_h)$ are postponed until the subsequent section. 

\section{The construction of the  nonconforming finite element Stokes complexes}

%define each space in \eqref{stokes-complex}.
% \begin{equation}\label{nonconforming-FE-derham-complex}
% \begin{tikzcd}
% 0 \arrow{r} & \mathbb R\arrow{r}{\subset}& \mathcal S_r^0(\mathcal T_h) \arrow{r}{\grad} & \mathcal S_{r-1,r}^{+,1}(\mathcal T_h)\arrow{r}{\curl} &  \mathcal S_{r-1}^{+,2}(\mathcal T_h) \arrow{r}{\div} & \mathcal S_{r-2}^3(\mathcal T_h) \arrow{r}{} & 0,
%  \end{tikzcd}
% \end{equation}
% where $r\geq 2$. 
% The space $\mathcal S_r^0(\mathcal T_h)$, which consists of piecewise $\mathcal S_r^0(K)$, is the serendipity finite element space introduced by \cite{arnold2011serendipity} (the DOFs will be given later), and $\mathcal S_{r-2}^3(\mathcal T_h)$ is simply the space of the discontinuous piecewise $P_{r-2}(K)$. 

%The main challenge lies in the construction of
 In this section, we construct 
 the $H(\grad\curl)$-nonconforming finite element space $\mathcal S_{r-1,r}^{+,1}(\mathcal T_h)$ and the $[H^1]^3$-nonconforming finite element space $\mathcal S_{r-1}^{+,2}(\mathcal T_h)$, and fit them into the  complex \eqref{nonconforming-FE-derham-complex}. To bound the consistency errors, the $\mathcal S_{r-1,r}^{+,1}(\mathcal T_h)$ and $\mathcal S_{r-1}^{+,2}(\mathcal T_h)$ are required to satisfy: 
 \begin{align}
 &\mathcal S_{r-1,r}^{+,1}(\mathcal T_h)\subseteq \{\bm{v}\in H(\operatorname{curl};\Omega): \langle[\![\operatorname{curl}\bm v\times \bm n_f]\!] , \bm q \rangle_f= 0~~\text{for all}~~\bm q \in [P_{r-2}(f)]^2, \, f\in \mathcal{F}_h^0\},  \label{regularity:curl}\\
 &\mathcal S_{r-1}^{+,2}(\mathcal T_h)\subseteq \{\bm{v}\in H(\operatorname{div};\Omega): \langle[\![\bm v\times \bm n_f]\!] , \bm q \rangle_f = 0~~\text{for all}~~\bm q \in [P_{r-2}(f)]^2, \, f\in \mathcal{F}_h^0\}.\label{regularity:div}
 \end{align}
% The polynomial space $\mathcal S_{r-1,r}^1(K)$ can be used to construct a family of $H(\curl)$-conforming finite element spaces \cite{gillette2019trimmed}, while the polynomial space $\mathcal S_{r-1}^2(K)$ can be used to define a family of $H(\div)$-conforming finite element spaces \cite{arnold2014finite}. We denote these spaces as $\mathcal S_{r-1,r}^1(\mathcal T_h)$ and $\mathcal S_{r-1}^2(\mathcal T_h)$, respectively, and their precise definitions will be given later.

To achieve the extra regularity required in \eqref{regularity:curl}, we can develop the finite element space $\mathcal S_{r-1,r}^{+,1}(\mathcal T_h)$ by enriching the $H(\curl)$-conforming finite element space $\mathcal S_{r-1,r}^1(\mathcal T_h)$ with newly constructed bubble functions. Furthermore, the $\operatorname{curl}$ of the bubble functions are used to enrich the $H(\div)$-conforming finite element space $\mathcal S_{r-1}^2(\mathcal T_h)$ so that the extra smoothness in \eqref{regularity:div} holds for $\mathcal S_{r-1}^{+,2}(\mathcal T_h)$.
% Thanks to \cite{gillette2019trimmed}, the polynomial space $\mathcal S_{r-1,r}^1(K)$ is able to yield a family of $H(\curl)$-conforming finite element spaces. And in light of \cite{arnold2014finite}, a family of $H(\div)$-conforming finite element spaces consisting of piecewise $\mathcal S_{r-1}^2(K)$ can be defined. 
% The finite element space $\mathcal S_{r-1,r}^{+,1}(\mathcal T_h)$ can be developed by enriching the $H(\curl)$-conforming finite element space with some newly constructed bubble functions to achieve the extra regularity in \eqref{regularity:curl}. Simultaneously, the space of the $\operatorname{curl}$ of the new bubble function space is supposed to be the bubble function space of the aforementioned $H(\div)$-conforming finite element space, and it brings the extra smoothness such that \eqref{regularity:div} holds for $\mathcal S_{r-1}^{+,2}(\mathcal T_h)$. 

 %Motivated by  \cite{johnny2012family,mardal2002robust, Tai2006ADD,xieXuDarcy2008}, the finite element spaces $\mathcal S_{r-1,r}^{+,1}(\mathcal T_h)$ and  $\mathcal S_{r-1}^{+,2}(\mathcal T_h)$ can be constructed by enriching the existing $H(\operatorname{curl};\Omega)$ and $H(\div)$-conforming finite element spaces with some bubble functions, respectively. The purpose of the 

 %删除Tai2006ADD 有没有构造前面的元

% 思路1 构造高阶非协调元很难。但一个derham复形简单，在复形的基础上可以构造一个加强版复形，这样就解决了咱们的问题。 （借助复形解决问题）
% 思路2 复形就是困难
%
% 简单预告
% 顺序：先谈最重要的两个元的构造
% 再说复形，证明

%% 高阶非协调元的构造本身困难，混合求解问题时，证明inf-sup困难。通过构造恰当的复形序列，为构造任意阶的非协调有限元提供了可能。复形序列的恰当结构在本文中扮演了重要的角色。

\subsection{The construction of the $H(\grad\curl)$-nonconforming elements --- $\mathcal S_{r-1,r}^{+,1}(\mathcal T_h)$}In this subsection, we construct $\mathcal S_{r-1,r}^{+,1}(\mathcal T_h)$ with regularity \eqref{regularity:curl}. 
The main idea is to enrich $\mathcal S_{r-1,r}^{1}(\mathcal T_h)$ by bubble functions constructed locally on each cube.
% In light of \cite[Theorem 4.2]{gillette2019trimmed}, the space $\mathcal S_{r-1,r}^1(K)$ can be used to construct the $H(\curl)$-conforming finite element on cubical grids. 
Therefore, let the shape function space of  $\mathcal S_{r-1,r}^{+,1}(\mathcal T_h)$ be of the form:
\begin{align*}
    \mathcal  S_{r-1,r}^{+,1}(K) =\mathcal S_{r-1,r}^{1}(K) +  \mathcal V^{r-2}(K),
\end{align*}
where $\mathcal S_{r-1,r}^{1}(K)$ is defined by \eqref{s1:onK} and $\mathcal V^{r-2}(K)$ is the bubble function space that is used to enforce the continuity of the moments
\begin{equation}\label{consistencyerror:gradcurl}
    \langle [\![\operatorname{curl}\bm{v}\times \bm{n}_f]\!], \bm{q}\rangle_f= 0 \quad \text{for all}~~\bm{q}\in [P_{r-2}(f)]^2,\, f\in \mathcal{F}_h^0. 
\end{equation}

Given a cube $K=(x_1^c-h_1,x_1^c+h_1)\times (x_2^c-h_2,x_2^c+h_2)\times (x_3^c-h_3,x_3^c+h_3)\in \mathcal{T}_h$, define
\begin{eqnarray*}
    B_K=\prod\limits_{i=1}^{3}(x_i-x_i^c+h_i)(x_i^c+h_i-x_i)&\text{and}& B_f=\frac{B_K}{\lambda_f},
\end{eqnarray*} 
where $\lambda_f$ is linear function such that $\lambda_f=0$ on $f$. An immediate result is 
\begin{align}\label{relation-gradb-bf}
	\grad B_K=-B_f\bm n_f\ \text{ on }f.
\end{align}
Then we define the bubble space $\mathcal V^{r-2}(K)$ by 
	\begin{align}
	\mathcal V^{r-2}(K)&=B_K\sum_{f\in \mathcal F_h(K)}B_f\mathcal V^{r-2}_f(K),\label{defV}
\end{align}
where 
\begin{align*}
    \mathcal V^{r-2}_f(K)&=\left\{\bm q\in [P_{r-2}(K)]^3\times \bm n_f:(\bm q,B_KB_f\bm w)_K=0,\ \forall  \bm w\in [P_{r-3}(K)]^3\times\bm n_f\right\}.
\end{align*}
For any bubble function $\bm z \in \mathcal V^{r-2}(K)$, it is straightforward that 
\begin{eqnarray}
     \bm z|_{\partial K} = 0 &\text{ and }&	(\bm z,\bm w)=0\text{ for any  }\bm w\in [P_{r-3}(K)]^3,\label{bubble-property:a2}
\end{eqnarray}
which yields the following direct sum decomposition. 
\begin{lemma}\label{direct:sum:curl}It holds that
    \[\mathcal  S_{r-1,r}^{+,1}(K) =\mathcal S_{r-1,r}^{1}(K) \oplus  \mathcal V^{r-2}(K).\]
\end{lemma}
\begin{proof}
If $\bm z \in \mathcal S_{r-1,r}^{1}(K) \cap  \mathcal V^{r-2}(K)$, the fact that $\mathcal S_{r-1,r}^{1}(K)\subseteq [P_{r+1}(K)]^3$ and $\bm z|_{\partial K}=0$ implies $\bm z = B_K \bm w$ with some $\bm w \in [P_{r-3}(K)]^3$, then \eqref{bubble-property:a2} concludes $\bm z =0$. 
\end{proof}

In addition, we present below a crucial property of the bubble function space  $\mathcal V^{r-2}(K)$.

\begin{lemma}\label{bubble-property:1}
    For $\bm z = B_K\sum\limits_{f\in \mathcal F_h(K)}B_f\bm q_f\in \mathcal V^{r-2}(K)$ with $\bm q_f\in \mathcal V_f^{r-2}(K)$, there holds
    \begin{equation*}
        (\curl \bm z\times\bm n_f)|_{f} = -B_f^2\bm q_f \quad \text{for all}\, f\in \mathcal{F}_h(K). 
    \end{equation*}
    Moreover, if $\bm q_f=0$ on $f$, then $\bm q_f=0$ in $K$.
\end{lemma}
 \begin{proof}
Using the product rule and \eqref{relation-gradb-bf}, we obtain
\begin{align*}
		(\curl \bm z)|_{\partial K}=&\Big(B_K\textstyle{\sum\limits_{f\in\mathcal F_h(K)}\curl(B_f\bm q_f)+\textstyle{\sum\limits_{f\in\mathcal F_h(K)}}B_f(\grad B_K)\times\bm q_f}\Big)\Big|_{\partial K}\\
		=&-\textstyle{\sum\limits_{f\in\mathcal F_h(K)}}B_f^2\bm n_f\times\bm q_f.
	\end{align*}
	% \begin{align*}
	% 	\left(\curl \bm z\times \bm n_f\right)|_f=&\Big(B_K\textstyle{\sum\limits_{f\in\mathcal F_h(K)}\curl(B_f\bm q_f)\times\bm n_f+B_f(\grad B_K)\times\bm q_f\times\bm n_f}\Big)\Big|_f\\
	% 	=&0-B_f^2\bm n_f\times\bm q_f\times\bm n_f|_f = -B_f^2 \bm q_f.
	% \end{align*}
The fact that $B_f = 0$ on $\partial K\backslash f$ yields
 \[(\curl \bm z\times\bm n_f)|_{f}= -B_f^2\bm n_f\times\bm q_f\times \bm n_f = -B_f^2\bm q_f.\]
    Furthermore, if $\bm q_f=0$ on $f$, then $\bm q_f=\lambda_f\bm p$ with $\bm p\in [P_{r-3}(K)]^3\times\bm n_f.$
	It follows from $B_K\bm p=B_f\bm q_f$
	that 
	\[0=(\bm q_f,B_fB_K\bm p)_K=(\bm q_f,B_f^2\bm q_f)_K,\]
	which implies $\bm q_f=0$ in $K$.
 \end{proof}
 
\begin{remark}\label{curl:norm:remark}
    Lemma \ref{bubble-property:1} implies $\|\curl\bm v\|_K$ defines a norm  for $\bm v\in \mathcal V^{r-2}(K)$. 
\end{remark}
 With the help of Lemma \ref{bubble-property:1}, we can prove the following result. 
\begin{lemma}\label{dim:bubble:V}
There holds
\[\dim \mathcal V^{r-2}(K)=\dim \curl \mathcal V^{r-2}(K) = 6(r-1)r.\]
\end{lemma}
\begin{proof}	
	Note that the dimension of $[P_{r-2}(K)]^3\times \bm n_f$ is $\frac{(r+1)(r-1)r}{3}$ and the dimension of $[P_{r-3}(K)]^3\times \bm n_f$ is $\frac{(r-2)(r-1)r}{3}$. Therefore, 
	\[\dim \mathcal V^{r-2}_f(K)=\frac{(r+1)(r-1)r-(r-2)(r-1)r}{3}=r(r-1),\]
	and hence $\dim \mathcal V^{r-2}(K)=6r(r-1)$.
 To prove $\dim \curl \mathcal V^{r-2}(K)=\dim \mathcal V^{r-2}(K)$, it suffices to prove the kernel of $\curl$ in $\mathcal V^{r-2}(K)$ is empty. To this end, let $\bm v=B_K\sum\limits_{f\in \mathcal F_h(K)}B_f\bm q_f$ with $\bm q_f\in \mathcal V_f^{r-2}(K)$ and suppose $\curl\bm v=0$. From Lemma \ref{bubble-property:1}, there holds $0=\big(\curl\bm v\times\bm n_f\big)\big|_f=-\big(B_f^2\bm q_f\big)\big|_f$ for all $f\in\mathcal F_h(K)$, which yields $\bm q_f=0$ on $f$, and hence $\bm q_f=0$ on $K$. Therefore, $\bm v=0$.
	\end{proof}

%%%%%%%%%%%%%%%%%%%%%%%%%%%%%

Next, the DOFs for $S_{r-1,r}^{+,1}(\mathcal T_h)$ ($r\geq 2$) are presented and proved to be unisolvent. 
For $\bm u\in \mathcal S_{r-1,r}^{+,1}(\mathcal T_h)$, the DOFs are given by
\begin{enumerate}\setcounter{enumi}{0}
	\item moments $\int_e \bm u\cdot\bm\tau_e q\d s,\ \forall q\in P_{r-1}(e)$ at all edges $e\in \mathcal E_h$,
	\item moments $\int_f \bm u\times\bm n_f \cdot \bm q\d A,\ \forall \bm q\in [P_{r-3}(f)]^2 \oplus \grad_f\widetilde{P}_{r-1}(f)$ at all faces $f\in \mathcal F_h$, 
	\item moments $\int_f \curl\bm u\times\bm n_f\cdot \bm q\d A,\ \forall \bm q\in [P_{r-2}(f)]^2$ at all faces $f\in \mathcal F_h$,  and
%	\item moments $\int_f \curl\bm u\cdot\bm nq\d A,\ \forall q\in P_{r-1}(f)\slash\mathbb R$ at all faces $f\in \mathcal F_h$,  and
	\item moments $\int_K \bm u\cdot\bm q\d V,\ \forall \bm q\in [P_{r-5}(K)]^3\oplus\curl[\widetilde{P}_{r-3}(K)]^3$ at all elements $K\in \mathcal T_h$.
\end{enumerate}

\begin{remark}\label{Hcurlconforming:unisolvence}
  According to \cite[Theorem 4.2]{gillette2019trimmed}, the DOFs (1), (2), and (4) on each cube $K\in\mathcal T_h$ are unisolvent for $\mathcal S_{r-1,r}^1(K)$, which leads to the $H(\curl)$-conforming finite element space $\mathcal S_{r-1,r}^1(\mathcal T_h)=\{\bm v\in  H(\operatorname{curl};\Omega):\bm v|_K\in \mathcal S^{1}_{r-1,r}(K)\text{ for all }K\in \mathcal T_h\}$.
\end{remark}

\begin{remark}\label{V-bubble}
   According to \eqref{bubble-property:a2}, the function in $\mathcal V^{r-2}(K)$ vanishes at the DOFs (1), (2), and (4). Therefore, they are indeed bubble functions of $\mathcal S_{r-1,r}^1(\mathcal T_h)$.
\end{remark}

  %%%%%%%%%%%%%%%%%%%%%%%%%%%%%%%
\begin{theorem}\label{th:unisolvence:curl}
The DOFs (1)--(4) for $\mathcal S_{r-1,r}^{+,1}(\mathcal T_h)$ are unisolvent. 
\end{theorem}
\begin{proof}
Lemma \ref{direct:sum:curl}, \eqref{dim:S1:K}, and Lemma \ref{dim:bubble:V} show that
\begin{equation*}
    \dim \mathcal  S_{r-1,r}^{+,1}(K)= \dim \mathcal S_{r-1,r}^{1}(K)+\dim \mathcal V^{r-2}(K)=\frac{1}{2}r^3 + \frac{17}{2}r^2+ 3r + 3, ~~\text{if}~~r\geq 3, 
\end{equation*}
and $\dim \mathcal  S_{1,2}^{+,1}(K)= 48$ if $r=2$. Thus the dimension of $\mathcal  S_{r-1,r}^{+,1}(K)$ is equal to the number of the DOFs (1)--(4) on each cube $K\in\mathcal{T}_h$. 
Suppose the DOFs (1)--(4) vanish for $\bm v\in \mathcal S_{r-1,r}^{+,1}(K)$, it suffices to prove $\bm{v}=0$. Lemma \ref{direct:sum:curl} implies $\bm v=\bm v_0+\bm v_b$ with $\bm v_0\in \mathcal S_{r-1,r}^{1}(K)$ and $\bm v_b\in \mathcal V^{r-2}(K)$. By Remark \ref{V-bubble}, the DOFs (1), (2), and (4) vanish for $\bm v_b$. 
Thus $\bm v_0$ vanishes at DOFs (1), (2), and (4) as well, which implies $\bm v_0=0$ by Remark \ref{Hcurlconforming:unisolvence}. Hence $\bm v = \bm v_b =B_K\sum\limits_{f\in\mathcal F_h(K)}B_f\bm q_f$ for some $\bm q_f\in \mathcal V_f^{r-2}(K)$. The vanishing DOFs (3) and Lemma \ref{bubble-property:1} lead to 
\begin{equation}
    0 = \int_f \operatorname{curl}\bm v\times \bm n_f \cdot \bm q_f\, dA = -\int_f B_f^2\bm q_f\cdot \bm q_f\, \d A \quad \text{for all}~~f \in \mathcal F_h(K).
\end{equation}
Therefore, we obtain $\bm q_f= 0$ on $f$ and consequently $\bm q_f = 0$ in $K$. Thus $\bm v= \bm 0$ follows.  
\end{proof}

The global $H(\grad\curl)$-nonconforming finite element space is defined by
\begin{align}
    \mathcal S_{r-1,r}^{+,1}(\mathcal T_h)=\{&\bm{v}\in H(\operatorname{curl};\Omega): \bm v|_K \in \mathcal S_{r-1,r}^{+,1}(K)~~\text{for all}~~K\in\mathcal T_h,\notag \\
    &\langle [\![\operatorname{curl}\bm v\times \bm{n}_f]\!]_f, \bm q \rangle_f=0~~\text{for all}~~\bm q \in [P_{r-2}(f)]^2,\, f\in \mathcal{F}_h^0\}.\label{global:curl:plus}
\end{align}

Suppose $\delta>0$. Denote by  $\bm\Pi_h^1$: $H^{\frac{1}{2}+\delta}(\curl;\Omega)\rightarrow \mathcal S_{r-1,r}^{+,1}(\mathcal T_h)$ the canonical interpolation operator based on the DOFs (1)--(4). We show the approximation property of $\bm\Pi_h^1$ below.

Thanks to \eqref{bubble-property:a2}, we can define the following global space:
% \begin{align}
%     \mathcal V^{r-2}(\mathcal T_h)=\{&\bm v\in  [H_0^1(\Omega)]^3:\bm v|_K\in \mathcal V^{r-2}(K)\text{ for all }K\in \mathcal T_h,\notag \\
%     &\text{the DOFs}\, (3)\, \text{are single-valued}\},
% \end{align}
%     \begin{align}
%     \mathcal S^{1}_{r-1,r}(\mathcal T_h)=\{&\bm v\in  H(\operatorname{curl};\Omega):\bm v|_K\in \mathcal S^{1}_{r-1,r}(K)\text{ for all }K\in \mathcal T_h, \notag\\
%     &\text{the DOFs (1)--(2) and (4) are single-valued}\}.\label{global:s1}
% \end{align}
\begin{align}
    \mathcal V^{r-2}(\mathcal T_h)=\{&\bm v\in  [H_0^1(\Omega)]^3:\bm v|_K\in \mathcal V^{r-2}(K)\text{ for all }K\in \mathcal T_h\}.\label{global:s1}
\end{align}
We define an interpolation $\bm\Pi_h^{\mathcal V}: H^{\frac{1}{2}+\delta}(\curl;\Omega)\rightarrow \mathcal V^{r-2}(\mathcal T_h)$ based on the DOF (3) as follows:
\begin{equation}\label{project:Hcurlbb:V}
 \langle\operatorname{curl}\bm \Pi_h^{\mathcal V} \bm v \times \bm n_f,\bm q\rangle_f=\langle\operatorname{curl}\bm v \times \bm n_f, \bm q\rangle_f ~~\text{for all}~~\bm q \in [P_{r-2}(f)]^2, f\in \mathcal F_h.
\end{equation}
For $\bm v \in H^{\frac{1}{2}+\delta}(\curl;\Omega)$, we have $\bm \Pi_h^{\mathcal V} \bm v = B_K\sum\limits_{f\in \mathcal F_h(K)}B_f \bm q_f$ with $\bm q_f \in \mathcal V_f^{r-2}(K)$. Substituting this into \eqref{project:Hcurlbb:V} and using Lemma \ref{bubble-property:1}, we get
\[\langle\operatorname{curl}\bm \Pi_h^{\mathcal V} \bm v \times \bm n_f,\bm q\rangle_f= -\langle B_f^2 \bm q_f, \bm q\rangle_f=\langle\curl\bm v \times \bm n_f, \bm q\rangle_f~~\text{for all}~~\bm q \in [P_{r-2}(f)]^2,\] 
which leads to $\|\bm q_f\|_{f} \leq C\|\curl\bm v\times \bm n_f\|_{f}$. By Remark \ref{curl:norm:remark} and the scaling argument, we obtain
\begin{equation}\label{project:estimate:Hcurlbb}
\begin{aligned}
    \|\bm\Pi_h^{\mathcal V}\bm v\|_K&\leq Ch_K \|\curl\bm\Pi_h^{\mathcal V}\bm v\|_K\leq Ch_K^{\frac{3}{2}}\sum_{f\in \mathcal F_h}\|\curl\bm\Pi_h^{\mathcal V}\bm v\times \bm n_f\|_{f}\leq Ch_K^{\frac{3}{2}}\sum_{f\in \mathcal F_h}\|B_f^2\bm q_f\|_{f}\\
    &\leq Ch_K^{\frac{3}{2}}\sum_{f\in \mathcal F_h} \|\operatorname{curl}\bm v \times \bm n_f\|_{f}\leq Ch_K(\|\curl\bm v\|_K+h_K\|\curl\bm v\|_{1,K}).
\end{aligned}  
\end{equation}
We also define $\bm\pi^1_h$: $H^{\frac{1}{2}+\delta}(\curl;\Omega)\rightarrow \mathcal S_{r-1,r}^{1}(\mathcal T_h)$ based on DOFs (1), (2), and (4). By the standard techniques in approximation analysis \cite[Theorem 5.41]{Monk2003},  the following approximation properties hold for $\bm v\in H^{s}(\curl;K), \ s\geq \frac{1}{2}+\delta$,
%For any $\bm v\in \mathcal S_{r-1,r}^{1}(K)$, it is easy to verify $\bm \pi^1_h \bm v = \bm v$. Thus the scaling argument shows the estimates
\begin{align}
	&\|\bm v-\bm \pi_h^{1}\bm v\|_{m,K}\leq Ch_K^{s-m}|\bm v|_{s,K}\text{ with } 0\leq m\leq s\leq r,\label{project:estimate:Hcurl:1}\\
	&\|\curl(\bm v-\bm \pi_h^{1}\bm v)\|_{m,K}\leq Ch_K^{s-m}|\curl\bm v|_{s,K}\text{ with } 0\leq m\leq s\leq r.\label{project:estimate:Hcurl:2}
\end{align}

Lemma \ref{direct:sum:curl} and the definitions of $\bm\Pi_h^1$, $\bm\pi_h^1$, and $\bm\Pi_h^{\mathcal V}$ imply
\begin{align}\label{identity}
\bm\Pi_h^1=\bm\pi_h^1+\bm\Pi_h^{\mathcal V}(\bm I-\bm\pi_h^1),
\end{align}
which, together with \eqref{project:estimate:Hcurlbb}--\eqref{project:estimate:Hcurl:2}, yield the approximation property of $\bm \Pi_h^1$ presented below. 
\begin{theorem}\label{est-pi1}
	Let $m$ be an integer and $s\geq 1$ satisfying $0\leq m\leq s\leq r$. Then for any $\bm v\in H^s(\curl;K)$, there holds
	\begin{align*}
	\|\bm v-\bm\Pi^1_h\bm v\|_{m,K}&\leq Ch_K^{s-m}|\bm v|_{s,K},\\
	\|\curl(\bm v-\bm\Pi^1_h\bm v)\|_{m,K}&\leq Ch_K^{s-m}|\curl\bm v|_{s,K}. 
	\end{align*}
\end{theorem}

%\begin{itemize}
%	\item $\int_e (\bm u-\bm\pi_h^1\bm u)\cdot\bm\tau_e q\d s=0,\ \forall q\in P_{r-1}(e)$ at all edges $e\in \mathcal E_h$,
%	\item $\int_f (\bm u-\bm\pi_h^1\bm u)\times\bm n \bm q\d A=0,\ \forall \bm q\in [P_{r-3}(f)]^2$ at all faces $f\in \mathcal F_h$, 
%%	\item $\int_f \curl(\bm u-\bm\Pi_h^{\mathcal S}\bm u)\times\bm n\cdot \bm q\d A,\ \forall \bm q\in [P_{r-2}(f)]^2$ at all faces $f\in \mathcal F_h$,  and
%	\item $\int_K (\bm u-\bm\pi_h^1\bm u)\cdot\bm q\d V=0,\ \forall \bm q\in [P_{r-5}(K)]^3$ at all elements $K\in \mathcal T_h$.
%\end{itemize}
%and
%\begin{itemize}
%	\item $\int_f \curl(\bm u-\bm\Pi_h^{\mathcal V}\bm u)\times\bm n_f\cdot \bm q\d %A=0,\ \forall \bm q\in [P_{r-2}(f)]^2$ at all faces $f\in \mathcal F_h$.
%\end{itemize}
%respectively. 

\subsection{The construction of the ${[H^1(\Omega)]}^3$-nonconforming elements --- $\mathcal S_{r-1}^{+,2}(\mathcal T_h)$ }We construct the $[H^1]^3$-nonconforming finite element $\mathcal S_{r-1}^{+,2}(\mathcal T_h)$ with regularity \eqref{regularity:div} in this subsection by enriching the $H(\div)$-conforming finite element space $\mathcal S_{r-1}^{2}(\mathcal T_h)$ with some bubbles.

Given a cube $K\in\mathcal{T}_h$, let the shape function space of $\mathcal S_{r-1}^{+,2}(\mathcal T_h)$ be
\begin{align*}
    \mathcal  S_{r-1}^{+,2}(K) =\mathcal S_{r-1}^{2}(K) + \mathcal U^{r-2}(K),
\end{align*}
where $S_{r-1}^{2}(K)$ is introduced in \eqref{def:s2} and \[\mathcal U^{r-2}(K)=\operatorname{curl}\mathcal V^{r-2}(K).\]

\iffalse
\begin{equation}\label{consistencyerror:gradcurl}
    \langle [\![\bm{v}\times \bm{n}_f]\!], \bm{q}\rangle_f= 0 \quad \text{for all}~~\bm{q}\in [P_{r-2}(f)]^2,\, f\in \mathcal{F}_h
\end{equation}
for $\bm v \in \mathcal S_{r-1}^{+,2}(\mathcal T_h)$.
\fi

We present below some key properties of $\mathcal U^{r-2}(K)$. 
\begin{lemma}\label{curlbubble-property}
	If $\bm z=\curl\bm v\in \mathcal U^{r-2}(K)$ with $\bm v\in \mathcal V^{r-2}(K)$, then 
	\begin{align}
		(\bm z\cdot \bm n_f)|_f&=0 \quad\text{ for all } f\in \mathcal F_h(K),\\
		(\bm z,\bm w)&=0\quad\text{ for all  }\bm w\in [P_{r-2}(K)]^3.
	\end{align}
		\end{lemma}
\begin{proof}
Since $\bm v|_{\partial K} = 0$, simple calculations lead to
\[\bm z\cdot \bm n_f|_f=\textstyle{\curl\bm v\cdot \bm n_f|_f=\div_f\big(\bm v\times\bm n_f\big)|_f}=0~~\text{for all}~~f\in \mathcal F_h(K).\]
Next, for $\bm w\in [P_{r-2}(K)]^3$, we use integration by parts, the fact that $\bm v|_{\partial K} = 0$, and \eqref{bubble-property:a2} to get
	\begin{align*}
	\textstyle{\big(\bm z,\bm w\big)}=\textstyle{\big(\curl\bm v,\bm w\big)}=\textstyle{\big(\bm v,\curl\bm w\big)}=0.
	\end{align*}
\end{proof}
Proceeding as the proof of Lemma \ref{direct:sum:curl}, with the help of Lemma \ref{curlbubble-property}, we can prove the following result. 
\begin{lemma}\label{direct:sum:div}It holds
    \[\mathcal S_{r-1}^{+,2}(K)=\mathcal S_{r-1}^{2}(K) \oplus \mathcal U^{r-2}(K).\]
\end{lemma}

From Lemma \ref{dim:bubble:V}, we have
\begin{align}\label{dim:u}
    \dim \mathcal U^{r-2}(K)=\dim \mathcal V^{r-2}(K)=6(r-1)r.
\end{align}
Therefore, the dimension of $\mathcal S_{r-1}^{+,2}(K)$ is
\begin{align}\label{dim-S2plus}
	\dim \mathcal S_{r-1}^{+,2}(K)= \dim \mathcal S_{r-1}^{2}(K)+\dim \mathcal U^{r-2}(K)=\frac{r(r^2+3r+8)}{2}+6r(r-1).
\end{align}

Next, the DOFs for $\mathcal S_{r-1}^{+,2}(\mathcal T_h)$ ($r\geq 2$) are presented and proved to be unisolvent. 
For $\bm u\in \mathcal S_{r-1}^{+,2}(\mathcal T_h)$, the DOFs are 
\begin{enumerate}\setcounter{enumi}{4}
	\item moments $\int_f \bm u\cdot\bm n_f  q\d A,\ \forall  q\in P_{r-1}(f)$ at all faces $f\in \mathcal F_h$,
	\item moments $\int_f \bm u\times\bm n_f \cdot \bm q\d A,\ \forall \bm  q\in [P_{r-2}(f)]^2$ at all faces $f\in \mathcal F_h$, and 
	\item moments $\int_K \bm u\cdot\bm q\d V,\ \forall \bm q\in [P_{r-3}(K)]^3$ at all elements $K\in\mathcal T_h$.
\end{enumerate}
\begin{remark}\label{unisolvence-div}
    In light of \cite[Theorem 3.6]{arnold2014finite}, DOFs (5) and (7) on each cube $K\in \mathcal T_h$ are unisolvent for the space $\mathcal S_{r-1}^{2}(K)$, which results in the $H(\div)$-conforming finite element space $\mathcal S_{r-1}^{2}(\mathcal T_h)= \{\bm v\in  H(\operatorname{div};\Omega):\bm v|_K\in \mathcal S^{2}_{r-1}(K)\text{ for all }K\in \mathcal T_h\}$.
\end{remark}
\begin{remark}\label{U-bubble}
    By Lemma \ref{curlbubble-property}, the functions in $\mathcal U^{r-1}(K)$ vanish at the DOFs (5) and (7). Then they are indeed bubbles functions of $\mathcal S_{r-1}^{2}(\mathcal T_h)$.
\end{remark}

\begin{theorem}\label{unisolvence-s2}
	The above DOFs (5)--(7) for $\mathcal S_{r-1}^{+,2}(\mathcal T_h)$ are unisolvent.
\end{theorem}
\begin{proof}
According to \eqref{dim-S2plus}, the dimension of $\mathcal S_{r-1}^{+,2}(K)$ is exactly the number of DOFs given by (5)--(7) on one element. Therefore, we only need to show that if the DOFs (5)--(7) vanish for $\bm v\in \mathcal S_{r-1}^{+,2}(K)$, then $\bm v\equiv 0$. To this end, we first write $\bm v = \bm v_0+\bm v_b$ with  $\bm v_0\in \mathcal S_{r-1}^{2}(K)$ and $\bm v_b\in \mathcal U^{r-2}(K)$.  It follows from  Lemma \ref{curlbubble-property} that the DOFs (5) and (7) vanish for $\bm v_b$, and hence for $\bm v_0$. Then $\bm v_0 = 0$ according Remark \ref{unisolvence-div}. Now $\bm v = \bm v_b =\curl\Big(B_K\sum\limits_{f\in\mathcal F_h(K)}B_f\bm q_f\Big)$ for some $\bm q_f\in \mathcal V_f^{r-2}(K)$. The vanishing DOFs (6) and Lemma \ref{bubble-property:1} lead to 
\begin{equation}
    0 = \int_f \bm v\times \bm n_f \cdot \bm q_f\, \d A = -\int_f B_f^2\bm q_f\cdot \bm q_f\, \d A \quad \text{for all}~~f \in \mathcal F_h(K).
\end{equation}
Therefore, we obtain $\bm q_f= 0$ on $f$ and consequently $\bm q_f = 0$ in $K$. Thus $\bm v= \bm 0$ follows. 
\end{proof}

The global $[H^1]^3$-nonconforming finite element space is defined by
\begin{align}
    \mathcal S_{r-1}^{+,2}(\mathcal T_h)=\{&\bm{v}\in H(\operatorname{div};\Omega): \bm v|_K \in \mathcal S_{r-1}^{+,2}(K)~~\text{for all}~~K\in\mathcal T_h,\notag \\
    &\langle [\![\bm v\times \bm{n}_f]\!]_f, \bm q \rangle_f=0~~\text{for all}~~\bm q \in [P_{r-2}(f)]^2,\, f\in \mathcal{F}_h^0\}.\label{global:div:plus}
\end{align}

The DOFs (5)--(7) naturally lead to the canonical interpolation $\bm\Pi^2_h$: $[H^{\frac{1}{2}+\delta}(\Omega)]^3\rightarrow \mathcal S_{r-1}^{+,2}(\mathcal T_h)$ with $\delta>0$. We show the approximation property of $\bm \Pi_h^2$ below. 

 Thanks to Lemma \ref{curlbubble-property}, we can define
\begin{align}
    \mathcal U^{r-2}(\mathcal T_h)=\{&\bm v\in   H_0(\div;\Omega):\bm v|_K\in \mathcal U^{r-2}(K)\text{ for all }K\in \mathcal T_h\}.\label{global:s2}
\end{align}
We then define an interpolation $\bm\Pi_h^{\mathcal U}: [H^{\frac{1}{2}+\delta}(\Omega)]^3\rightarrow \mathcal  U^{r-2}(\mathcal T_h)$ with $\delta>0$ based on the DOFs (6) as follows:
\begin{equation}\label{project:Hdivbb:U}
 \langle\bm \Pi_h^{\mathcal U} \bm v \times \bm n_f,\bm q\rangle_f=\langle\bm v \times \bm n_f, \bm q\rangle_f ~~\text{for all}~~\bm q \in [P_{r-2}(f)]^2\text{ and } f\in \mathcal F_h.
\end{equation}
Let $\bm v = \operatorname{curl}\bm w$ with $\bm w\in H^{\frac{1}{2}+\delta}(\operatorname{curl};\Omega)$. By \eqref{project:Hdivbb:U} and \eqref{project:Hcurlbb:V}, there holds
\begin{equation*}
 \langle\bm \Pi_h^{\mathcal U} \operatorname{curl}\bm w \times \bm n_f,\bm q\rangle_f=\langle\operatorname{curl}\bm w \times \bm n_f, \bm q\rangle_f =\langle\operatorname{curl}\bm \Pi_h^{\mathcal V} \bm w \times \bm n_f,\bm q\rangle_f  ~~\text{for all}~~\bm q \in [P_{r-2}(f)]^2\text{ and } f\in \mathcal F_h,
\end{equation*}
This implies a commuting property 
\begin{equation}\label{commute:U:V}
 \bm \Pi_h^{\mathcal U} \operatorname{curl}\bm w =\operatorname{curl}\bm \Pi_h^{\mathcal V} \bm w,
\end{equation}
which, together with  the boundedness of $\bm \Pi_h^{\mathcal V}$ \eqref{project:estimate:Hcurlbb}, yields
\begin{equation}\label{project:estimate:Hdivbb}
    \|\bm \Pi_h^{\mathcal{U}}\bm v\|_{K}\leq C(\|\bm v\|_K+h_K\|\bm v\|_{1,K})~~\text{for}~~\bm v\in [H^1(K)]^3.
\end{equation}
In addition, we define $\bm\pi^2_h$: $[H^{\frac{1}{2}+\delta}(\Omega)]^3\rightarrow \mathcal S_{r-1}^{2}(\mathcal T_h)$ with $\delta> 0$ based on DOFs (5) and (7). For $\bm v\in [H^{s}(K)]^3, \ s\geq \frac{1}{2}+\delta$, the estimate
\begin{align}
    \|\bm v-\bm\pi_h^{2}\bm v\|_{m,K}\leq Ch_K^{s-m}|\bm v|_{s,K}~~\text{ with }  0\leq m\leq s\leq r\label{project:estimate:Hdiv:1}
\end{align}
follows by the scaling argument.

Lemma \ref{direct:sum:div} and the definitions of 
$\bm\Pi_h^2$, $\bm\pi_h^2$, and $\bm\Pi_h^{\mathcal U}$ lead to 
\begin{align}\label{identity:2}
\bm\Pi_h^2=\bm\pi_h^2+\bm\Pi_h^{\mathcal U}(\bm I-\bm\pi_h^2).
\end{align}
The approximation property of $\bm \Pi_h^2$ can be derived by \eqref{project:estimate:Hdivbb}, \eqref{project:estimate:Hdiv:1}, and \eqref{identity:2}.
\begin{theorem}\label{err:interpolation:2h}
	Let $m$ be an integer and $s\geq 1$ satisfying $0\leq m\leq s\leq r$. Then for any $\bm v\in [H^s(K)]^3$, there holds
	\[\|\bm v-\bm\Pi^2_h\bm v\|_{m,K}\leq Ch_K^{s-m}|\bm v|_{s,K}.\]
\end{theorem}

%%%%%%%%%%%%%%%%%%%%%%%%%%%%%%%%%%%%%%%
\subsection{The exact sequence}
In this subsection, we present the precise definition of the spaces $\mathcal S_{r}^{0}(\mathcal T_h)$ and $\mathcal S_{r-2}^{3}(\mathcal T_h)$, and prove the exactness of complex \eqref{nonconforming-FE-derham-complex}. 

For $K\in\mathcal T_h$, the shape function space of $\mathcal S_r^0(\mathcal T_h)$ is $\mathcal S_r^0(K)$, which has been defined in Section 2. 
The DOFs for $u\in \mathcal S_r^0(\mathcal T_h)$ \cite[(2.2)]{arnold2011serendipity} are 
\begin{enumerate}\setcounter{enumi}{7}
	\item function values $u(v)$ at all vertices $v\in\mathcal V_h$,
	\item moments $\int_e uq\d s,\ \forall q\in P_{r-2}(e)$ at all edges $e\in \mathcal E_h$,
	\item moments $\int_f uq\d A,\ \forall q\in P_{r-4}(f)$ at all faces $f\in \mathcal F_h$,  and
	\item moments $\int_K uq\d V,\ \forall q\in P_{r-6}(K)$ at all elements $K\in \mathcal T_h$.
\end{enumerate}
Define
\[\mathcal S_{r-2}^3(\mathcal T_h) = \{u\in L^2(\Omega):u|_K\in P_{r-2}(K)\}.\]
Let $\bm \pi_h^3:L^2(\Omega)\rightarrow\mathcal S_{r-2}^3(\mathcal T_h) $ denote the $L^2$ projection. 

We prove below that the spaces $\mathcal S_r^0(\mathcal T_h)$,  $\mathcal S_{r-1,r}^1(\mathcal T_h)$, $\mathcal S_{r-1}^2(\mathcal T_h)$, and $\mathcal S_{r-2}^3(\mathcal T_h)$ form an exact complex. To this end, we first present the following two lemmas.

\begin{lemma}\label{commu-23}
The projections $\bm \Pi_h^2$, $\bm \pi_h^2$ and $\bm \pi_h^3$ satisfy
\begin{align}   
   \operatorname{div}\bm\Pi^2_h \bm{v}=\operatorname{div}\bm\pi^2_h \bm{v}=\bm\pi^3_{h} \operatorname{div}\bm{v},\quad \text{for all }~~\bm{v}\in [H^{\frac{1}{2}+\delta}(\Omega)]^{3}.  \label{commute:a2}
\end{align}   
\end{lemma}
\begin{proof}
The identity \eqref{identity:2} shows $\operatorname{div}\bm\Pi^2_h \bm{v}=\operatorname{div}\bm\pi^2_h \bm{v}$. It suffices to prove the commuting property $\operatorname{div}\bm\Pi^2_h \bm{v}=\bm\pi^3_{h} \operatorname{div}\bm{v}$. Given $\bm v \in[H^{\frac{1}{2}+\delta}(\Omega)]^{3}$, we have $\operatorname{div}\bm \Pi_h^2 \bm v - \bm \pi_h^3\operatorname{div}\bm v \in \mathcal S_{r-2}^3(\mathcal T_h)$. Then for any $K\in\mathcal T_h$ and $q\in  P_{r-2}(K)$, we apply integration by parts and the definitions of $\bm \pi_h^3$ and $\bm \Pi_h^2$ to get
\begin{align*}
    &(\operatorname{div}\bm \Pi_h^2 \bm v - \bm \pi_h^3\operatorname{div}\bm v, q)_K= \sum_{f\in \mathcal F_h(K)}\langle\bm \Pi_h^2 \bm v\cdot \bm n_{f,K}, q\rangle_{f}- (\bm \Pi_h^2 \bm v, \operatorname{grad}q)_K - (\operatorname{div}\bm v, q)_K\\
     =&\sum_{f\in \mathcal F_h(K)}\langle \bm v\cdot \bm n_{f,K}, q\rangle_{f}- ( \bm v, \operatorname{grad}q)_K - (\operatorname{div}\bm v, q)_K= 0.
\end{align*}
Thus \eqref{commute:a2} follows.
\end{proof}
Lemma \ref{commu-23} implies the following crucial result. 
\begin{lemma}\label{onto:2:3}
For $r\geq 2$, there holds
\begin{equation}
    \operatorname{div}\mathcal S_{r-1}^{+,2}(\mathcal T_h) = \operatorname{div}\mathcal S_{r-1}^2(\mathcal T_h) = \mathcal S_{r-2}^3(\mathcal T_h).
\end{equation}
\end{lemma}
\begin{proof}
It suffices to prove $\mathcal S_{r-2}^3(\mathcal T_h)\subset \operatorname{div}\mathcal S_{r-1}^{+,2}(\mathcal T_h)$ and $\mathcal S_{r-2}^3(\mathcal T_h)\subset \operatorname{div}\mathcal S_{r-1}^{2}(\mathcal T_h)$ since the opposite inclusion is trivial. 
From the exactness of the continuous Stokes complex \eqref{stokes-complex}, for any $q_h\in {\mathcal S}_{r-2}^3(\mathcal T_h)\subset L^2(\Omega)$, there exists $\bm w\in  [H^1(\Omega)]^3$ satisfying $\div\bm w= q_h$. By \eqref{commute:a2}, we have $q_h=\bm\pi_h^{3} q_h=\bm\pi_h^{3}\div \bm w=\div\bm\pi_h^2 \bm w\in \div {\mathcal S}_{r-1}^2(\mathcal T_h)$, and similarly $q_h=\div\bm\Pi_h^2 \bm w\in \div {\mathcal S}_{r-1}^{+,2}(\mathcal T_h)$.
\end{proof}

\iffalse
And define the following global spaces:
\begin{align*}
&\mathring{\mathcal S}_{r}^{0}(\mathcal T_h)=\mathcal S_{r}^{0}(\mathcal T_h)\cap H_0^1(\Omega),&& \mathring{\mathcal S}^{1}_{r-1,r}(\mathcal T_h)=\mathcal S_{r-1,r}^1(\mathcal T_h)\cap H_0(\operatorname{curl};\Omega),\\
&\mathring{\mathcal S}^{2}_{r-1}(\mathcal T_h)=\mathcal S_{r-1}^2(\mathcal T_h)\cap H_0(\operatorname{div};\Omega),&&
\mathring{\mathcal S}_{r-2}^{3}(\mathcal T_h) = {\mathcal S}_{r-2}^{3}(\mathcal T_h)\cap L_0^2(\Omega).
\end{align*}
\fi

\begin{lemma}\label{exactness}
For $r\geq 2$, the finite element sequence in \eqref{FE-derham-complex} 
% \begin{equation}\label{FE-derham-complex-0}
% \begin{tikzcd}
%  \mathbb{R}\arrow{r}{\subset}& {\mathcal S}_r^0(\mathcal T_h) \arrow{r}{\grad} & {\mathcal S}_{r-1,r}^1(\mathcal T_h)\arrow{r}{\curl} &  {\mathcal S}_{r-1}^2(\mathcal T_h) \arrow{r}{\div} & {\mathcal S}_{r-2}^3(\mathcal T_h) \arrow{r}{} & 0
%  \end{tikzcd}
% \end{equation}
is an exact complex when $\Omega$ is contractible. 
\end{lemma}
\begin{proof}
We show the exactness at each space. 
The exactness at ${\mathcal S}_{r}^0(\mathcal T_h)$ is trivial as the kernel of the operator $\operatorname{grad}$ only consists of constant functions. The exactness at ${\mathcal S}_{r-1,r}^1(\mathcal T_h)$ follows from \cite[Theorem 3.5]{gillette2019trimmed}. Besides, Lemma \ref{onto:2:3} shows the exactness at ${\mathcal S}_{r-2}^3(\mathcal T_h)$. To show the exactness at ${\mathcal S}_{r-1}^1(\mathcal T_h)$, it suffices to prove $\operatorname{ker}(\operatorname{div})\cap \mathcal S_{r-1}^2(\mathcal T_h)=\operatorname{curl} \mathcal{S}_{r-1,r}^1(\mathcal T_h)$. It is easy to verify 
\begin{equation*}
    \operatorname{curl} \mathcal{S}_{r-1,r}^1(\mathcal T_h)\subseteq \operatorname{ker}(\operatorname{div})\cap \mathcal S_{r-1}^2(\mathcal T_h).
\end{equation*}
To prove $\operatorname{ker}(\operatorname{div})\cap \mathcal S_{r-1}^2(\mathcal T_h)=\operatorname{curl} \mathcal{S}_{r-1,r}^1(\mathcal T_h)$, we only need to show the dimensions of the two spaces are equal. 
Let $\mathcal V_0$, $\mathcal E_0$, $\mathcal F_0$, and $\mathcal K_0$ denote the number of all vertices, edges, faces, and 3D cells. We have
	 \begin{align*}
	 	\dim{\mathcal S}_r^0(\mathcal T_h) &= \mathcal V_0+(r-1)\mathcal E_0+\frac{(r-3)(r-2)}{2}\mathcal F_0+\frac{(r-5)(r-4)(r-3)}{6}\mathcal K_0,\\
	 	\dim{\mathcal S}_{r-1,r}^1(\mathcal T_h)& = r\mathcal E_0+{(r^2 - 2r + 2)}\mathcal F_0+\frac{r^3 - 7r^2 + 18r - 18}{2}\mathcal K_0,\\
	 	\dim{\mathcal S}_{r-1}^2(\mathcal T_h) &= \frac{(r+1)r}{2}\mathcal F_0+\frac{(r-1)(r-2)r}{2}\mathcal K_0,\\
	 	\dim{\mathcal S}_{r-2}^3(\mathcal T_h)& = \frac{(r+1)(r-1)r}{6}\mathcal K_0.
	 \end{align*}
Then the dimensions of $\operatorname{curl} \mathcal{S}_{r-1,r}^1(\mathcal T_h)$ and $\operatorname{ker}(\operatorname{div})\cap \mathcal S_{r-1}^2(\mathcal T_h)$ read
\begin{align*}
 \operatorname{dim}\operatorname{curl} \mathcal{S}_{r-1,r}^1(\mathcal T_h) &= \operatorname{dim}\mathcal{S}_{r-1,r}^1(\mathcal T_h) - \operatorname{dim} \mathcal S_r^0(\mathcal T_h)+ 1\\
 &= -\mathcal V_0+\mathcal E_0+(\frac{1}{2}r^2+\frac{1}{2}r-1)\mathcal F_0 +(\frac{1}{3}r^3-\frac{3}{2}r^3+\frac{7}{6}r+1) \mathcal K_0+1,\\
\operatorname{dim}\operatorname{ker}(\operatorname{div})\cap \mathcal S_{r-1}^2(\mathcal T_h) &= \operatorname{dim}\mathcal S_{r-1}^2(\mathcal T_h)-\operatorname{dim} S^3_{r-2}(\mathcal T_h)\\
&= (\frac{1}{2}r^2+\frac{1}{2}r)\mathcal F_0 +(\frac{1}{3}r^3-\frac{3}{2}r^3+\frac{7}{6}r+1) \mathcal K_0,
\end{align*}
which, together with Euler's formula $\mathcal V_0-\mathcal E_0+\mathcal F_0-\mathcal K_0+1=0$, yields $\operatorname{dim}\operatorname{curl} \mathcal{S}_{r-1,r}^1(\mathcal T_h)=\operatorname{dim}\operatorname{ker}(\operatorname{div})\cap \mathcal S_{r-1}^2(\mathcal T_h)$.
\end{proof}

\iffalse
\begin{align*}
&\mathring{\mathcal S}_{r}^{0}(\mathcal T_h)=\{\bm u\in {\mathcal S}_{r}^{0}(\mathcal T_h): \text{the boundary DOFs (8)--(10) vanish}\},\\
	&\mathring{\mathcal S}_{r-1,r}^{+,1}(\mathcal T_h)=\{\bm u\in {\mathcal S}_{r-1,r}^{+,1}(\mathcal T_h): \text{the boundary DOFs (1)--(3) vanish}\},\\
	&\mathring{\mathcal S}_{r-1}^{+,2}(\mathcal T_h)=\{\bm u\in {\mathcal S}_{r-1}^{+,2}(\mathcal T_h): \text{the boundary DOFs of (5)--(6) vanish}\},\\
 &\mathring{\mathcal S}_{r-2}^{3}(\mathcal T_h) = {\mathcal S}_{r-2}^{3}(\mathcal T_h)\cap L_0^2(\Omega).
\end{align*}
\fi 

We prove the exactness of the nonconforming finite element complex \eqref{nonconforming-FE-derham-complex} below.

\begin{theorem}\label{exactness1}
For $r\geq 2$, the nonconforming finite element complex \eqref{nonconforming-FE-derham-complex}
% \begin{equation}\label{FE-derham-complex-0}
% \begin{tikzcd}
%  \mathbb{R}\arrow{r}{\subset}& {\mathcal S}_r^0(\mathcal T_h) \arrow{r}{\grad} & {\mathcal S}_{r-1,r}^{+,1}(\mathcal T_h)\arrow{r}{\curl} &  {\mathcal S}_{r-1}^{+,2}(\mathcal T_h) \arrow{r}{\div} & {\mathcal S}_{r-2}^3(\mathcal T_h) \arrow{r}{} & 0
%  \end{tikzcd}
% \end{equation}
is an exact complex when $\Omega$ is contractible. 
\end{theorem}
\begin{proof}
It is trivial to show the exactness at ${\mathcal S}_r^0(\mathcal T_h)$. 
The exactness at $\mathcal S_{r-2}^3(\mathcal T_h)$ follows from Lemma \ref{onto:2:3}. To show the exactness at $\mathcal S_{r-1}^{+,2}(\mathcal T_h)$, for $\bm w_h\in {\mathcal S}_{r-1}^{+,2}(\mathcal T_h)$ with $\div \bm w_h = 0$, it suffices to prove there exists some $ \bm v_h\in {\mathcal S}_{r-1,r}^{+,1}(\mathcal T_h)$ such that $\bm w_h=\curl \bm v_h$. In light of the definition of ${\mathcal S}_{r-1}^{+,2}(\mathcal T_h)$, we have
	\[\bm w_h=\bm w_0+\curl\bm v_1\text{ with }\bm w_0\in {\mathcal S}_{r-1}^2(\mathcal T_h)\text{ and }\bm v_1\in \mathcal V^{r-2}(\mathcal T_h).\]
Then $\div\bm w_h=0$ leads to $\div\bm w_0=0.$ By Lemma \ref{exactness}, there exists $\bm v_0\in \mathcal S_{r-1,r}^1(\mathcal T_h)$ such that $\bm w_0=\curl \bm v_0.$
	Define $\bm v_h:=\bm v_0+\bm v_1$, then $\bm w_h=\curl\bm v_h.$
	Since $\bm w_h\in {\mathcal S}_{r-1}^{+,2}(\mathcal T_h)$, we have $\langle[\![\bm w_h\times \bm n_f]\!]_f,\bm q\rangle_f =0$, and hence 
	 $\langle[\![\curl\bm v_h\times \bm n_f]\!]_f,\bm q\rangle_f=0$, for all $\bm q\in [P_{r-3}(f)]^2$ and $f\in\mathcal{F}_h$. Thus we obtain $\bm v_h\in {\mathcal S}_{r-1,r}^{+,1}(\mathcal T_h)$ and consequently the exactness at $\mathcal S_{r-1}^{+,2}(\mathcal T_h)$. To prove the exactness at ${\mathcal S}_{r-1,r}^{+,1}(\mathcal T_h)$, we need to show $\operatorname{ker}(\operatorname{curl})\cap \mathcal S_{r-1,r}^{+,1}(\mathcal T_h)=\operatorname{grad}\mathcal S_{r}^0(\mathcal T_h)$. It is easy to confirm $\operatorname{grad}\mathcal S_{r}^0(\mathcal T_h)\subseteq\operatorname{ker}(\operatorname{curl})\cap \mathcal S_{r-1,r}^{+,1}(\mathcal T_h)$. To show $\operatorname{grad}\mathcal S_{r}^0(\mathcal T_h)\subseteq \operatorname{ker}(\operatorname{curl})\cap \mathcal S_{r-1,r}^{+,1}(\mathcal T_h)$, we count dimensions:
  \begin{align*}
  \operatorname{dim}\operatorname{ker}(\operatorname{curl})\cap \mathcal S_{r-1,r}^{+,1}(\mathcal T_h)& = \operatorname{dim} \mathcal S_{r-1,r}^{+,1}(\mathcal T_h)-\operatorname{dim}\operatorname{curl}\mathcal S_{r-1,r}^{+,1}(\mathcal T_h)\\
  &=\operatorname{dim} \mathcal S_{r-1,r}^{+,1}(\mathcal T_h)-\operatorname{dim}\mathcal S_{r-1}^{+,2}(\mathcal T_h) + \operatorname{dim}\mathcal S_{r-2}^3(\mathcal T_h)\\
  & =\operatorname{dim} \mathcal S_{r-1,r}^{1}(\mathcal T_h)-\operatorname{dim}\mathcal S_{r-1}^{2}(\mathcal T_h) + \operatorname{dim}\mathcal S_{r-2}^3(\mathcal T_h)\\
  &=\operatorname{dim} \operatorname{grad}\mathcal S_{r}^0(\mathcal T_h).
  \end{align*}
  Lemma \ref{dim:bubble:V} and Lemma \ref{exactness} account for the last two equalities, respectively. 
  \end{proof}
We define finite element spaces with vanishing traces on $\partial \Omega$.
\begin{align*}
\mathring{\mathcal S}_{r}^{0}(\mathcal T_h)&=\{ w\in  H_0^1(\Omega):\, w|_{K}\in \mathcal S_{r}^{0}(K)\ \text{for all}~K\in\mathcal{T}_{h}\},\\
\mathring{\mathcal S}_{r-1,r}^{+,1}(\mathcal T_h)&=\{\bm{w}\in  H_{0}(\curl;\Omega):\, \bm{w}|_{K}\in \mathcal S_{r-1,r}^{+,1}(K)\ \text{for all}~K\in\mathcal{T}_{h},  \\
&\langle [\![\curl\bm{w}\times \bm{n}_f]\!]_f, \bm{q}\rangle_{f}=0\ \text{for all}~\bm{q}\in [P_{r-2}(f)]^{2}\text{ and } f\in\mathcal{F}_{h}\},\\
\mathring{\mathcal S}_{r-1}^{+,2}(\mathcal T_h)&=\{ \bm{v}\in  H_{0}(\operatorname{div};\Omega):\, \bm{v}|_{K}\in \mathcal S_{r-1}^{+,2}(K)\ \text{for all}~K\in\mathcal{T}_{h},  \\
&\langle [\![\bm{v}\times \bm{n}_f]\!]_f, \bm{q}\rangle_{f}=0\ \text{for all}~\bm{q}\in [P_{r-2}(f)]^{2}\text{ and } f\in\mathcal{F}_{h}\},\\
\mathring{\mathcal S}_{r-2}^{3}(\mathcal T_h)&=\{ w\in  L_0^2(\Omega):\, w|_{K}\in \mathcal S_{r-2}^{3}(K)\ \text{for all}~K\in\mathcal{T}_{h}\},
\end{align*}

\begin{theorem}\label{exactness2}
    For $r\geq 2$, the following sequence
\begin{equation}
\begin{tikzcd}
 0\arrow{r}{\subset}& \mathring{\mathcal S}_r^0(\mathcal T_h) \arrow{r}{\grad} & \mathring{\mathcal S}_{r-1,r}^{+,1}(\mathcal T_h)\arrow{r}{\curl_h} &  \mathring{\mathcal S}_{r-1}^{+,2}(\mathcal T_h) \arrow{r}{\div_h} & \mathring{\mathcal S}_{r-2}^3(\mathcal T_h) \arrow{r}{} & 0
 \end{tikzcd}
\end{equation}
is an exact complex when $\Omega$ is contractible. 
\end{theorem}
\begin{proof}
    The proof of this result is quite similar to that of Theorem \ref{exactness1} and so is omitted.
\end{proof}

\section{Applications to $-\curl\Delta\curl$ Problem}
In this section, we apply $\mathring{\mathcal S}_{r}^{0}(\mathcal T_h)$ and $\mathring{\mathcal S}_{r-1,r}^{+,1}(\mathcal T_h)$ to the quad-curl problem \eqref{prob1}. To simplify notation, we use $\mathring{\mathcal S}_{r}^{0}$ and $\mathring{\mathcal S}_{r-1,r}^{+,1}$ in the following. 
We define $H_0(\grad\curl;\Omega)$ as the space of functions in $H(\grad\curl;\Omega)$ with vanishing boundary conditions:
\begin{align*}
H_0(\grad\curl;\Omega):=&\left\{\bm u \in H(\grad\curl;\Omega): {\bm n}\times\bm u=0\; \text{and}\ \curl\bm u=0\; \text{on}\ \partial \Omega\right\}.
\end{align*}
%The space of $\bm L^2(D)$ functions with a divergence zero is denoted by $H(\div^0;D)$ and defined by
%\[H(\text{div}^0;D) :=\{\bm u\in {\bm L}^2(D):\; \nabla\cdot \bm u=0\}.\]
%For the sake of satisfying divergence-free condition, we adopt mixed methods where the constraint $\div\bm u=0$ in \eqref{prob1} is satisfied in a weak sense by introducing an auxiliary unknown $p$ and employing a mixed
We use a Lagrange multiplier $p$ to deal with the divergence-free condition. 
The mixed variational formulation of \eqref{prob1} is to find $(\bm u;p)\in H_0(\grad\curl;\Omega)\times H^1_0(\Omega)$  such that
\begin{equation}\label{prob2}
\begin{aligned}
a_1(\bm u,\bm v)+b_1(\bm v,p)&=(\bm f, \bm v)~~&&\forall \bm v\in H_0(\grad\curl;\Omega),\\
b_1(\bm u,q)&=0~~&&\forall q\in H_0^1(\Omega),
\end{aligned}
\end{equation}
with $a_1(\bm u,\bm v):=\mu(\grad\curl\bm u, \grad\curl\bm v) +(\curl\bm u,\curl\bm v)+ \gamma(\bm u,\bm v)$ and $b_1(\bm v,p) = (\bm v,\grad p).$

%The weak form \eqref{prob22} can be regarded as a model problem for the high order problems in MHD, e.g., \cite[(1)]{chacon2007steady} and continuum mechanics with size effects, e.g., \cite[(3.27)]{mindlin1962effects}, \cite[(35)]{park2008variational}. 

Taking $\bm v=\grad p$ in \eqref{prob2} and using the vanishing boundary condition of $p$, we can obtain that $p\equiv 0$.

% We take 
% \begin{equation*}
% \begin{split}
% \bm{W}_{h}=\mathring{\mathcal S}_{r-1,r}^{+,1}(\mathcal T_h)&=\{\bm{w}\in  H_{0}(\curl;\Omega):\, \bm{w}|_{K}\in \mathcal S_{r-1,r}^{+,1}(K)\ \text{for all}~K\in\mathcal{T}_{h},  \\
% &\langle [\![\curl\bm{w}\times \bm{n}_f]\!]_f, \bm{q}\rangle_{f}=0\ \text{for all}~\bm{q}\in [P_{r-2}(f)]^{2}\text{ and } f\in\mathcal{F}_{h}\},\\
% \mathring{\mathcal S}_{r}^0=\mathring{\mathcal S}_{r}^{0}(\mathcal T_h)&=\{ w\in  H_0^1(\Omega):\, w|_{K}\in \mathcal S_{r}^{0}(K)\ \text{for all}~K\in\mathcal{T}_{h}\},
% \end{split}
% \end{equation*}
% for $r\geq 2$.

The nonconforming finite element method for \eqref{prob2} seeks $(\bm u_h;p_h)\in \mathring{\mathcal S}_{r-1,r}^{+,1} \times \mathring{\mathcal S}_{r}^0$  such that
\begin{equation}\label{prob22}
\begin{aligned}
a_{1,h}(\bm u_h,\bm v_h)+b_{1}(\bm v_h,p_h)&=(\bm f, \bm v_h)~~&&\forall \bm v_h\in \mathring{\mathcal S}_{r-1,r}^{+,1},\\
b_{1}(\bm u_h,q_h)&=0~~&&\forall q_h\in \mathring{\mathcal S}_{r}^0,
\end{aligned}
\end{equation}
where $a_{1,h}(\bm u_h,\bm v_h):=\mu\sum\limits_{K\in\mathcal T_h}(\grad\curl\bm u_h, \grad\curl\bm v_h)_K +(\curl\bm u_h,\curl\bm v_h)+\gamma (\bm u_h,\bm v_h)$.

%The finite element spaces $\mathring{\mathcal S}_{r-1,r}^{+,1}$ and  $\mathring{\mathcal S}_{r}^0$ should satisfy:
%\begin{itemize}
%	\item the inf-sup condition
%\begin{align}\label{infsup2}
%	\sup_{\bm v\in\mathring{\mathcal S}_{r-1,r}^{+,1}}\frac{b_1(\bm v,w)}{\|\bm v\3bar_1}\geq C\|w\|_1, \quad\forall w\in \mathring{\mathcal S}_{r}^0,
%\end{align}
%\item the coercivity condition
%\begin{align}\label{coer}
%\|\bm v_h\|\gc,h_{1,h}}\geq c\|\bm v_h\3bar_1\text{ for }\bm v_h\in N_h,
%\end{align}
%\item discrete divergence-free condition
%\begin{align}\label{div-free3}
%\grad \mathring{\mathcal S}_{r}^0=\ker(\curl)\cap \mathring{\mathcal S}_{r-1,r}^{+,1}.
%\end{align} 
%\end{itemize} 

%By Theorem \ref{exactness1},
%\begin{align}
%\grad \mathring{\mathcal S}_{r}^0=\ker(\curl)\cap \mathring{\mathcal S}_{r-1,r}^{+,1}.
%\end{align}
For $\bm v_h\in \mathring{\mathcal S}_{r-1,r}^{+,1}+H(\grad\curl;\Omega)$, define 
\begin{eqnarray*}
	\3bar\bm v_h\3bar_1^2=\|\bm v_h\|^2+\|\curl\bm v_h\|^2+\mu|\curl\bm v_h|^2_{1,h},
\end{eqnarray*}
where $|\cdot|_{1,h} = \sum\limits_{K\in\mathcal T_h}|\cdot|_{1,K}$.
It is clear that
\begin{align}\label{bounda}
	&a_{1,h}(\bm u_h,\bm v_h)\leq C\3bar\bm v_h\3bar_1\3bar\bm u_h\3bar_1~~&&\text{for }\bm u_h,\bm v_h\in\mathring{\mathcal S}_{r-1,r}^{+,1}+H(\grad\curl;\Omega),\\
	&b_{1}(\bm v_h, q_h)\leq C\3bar\bm v_h\3bar_1\|q_h\|_{1}~~&&\text{for }\bm v_h\in\mathring{\mathcal S}_{r-1,r}^{+,1}+H(\grad\curl;\Omega) \text{ and }q_h\in H^1(\Omega).\label{boundb}
\end{align}
The inf-sup condition holds since for any $q_h\in \mathring{\mathcal S}_{r}^0$, 
\begin{align}\label{infsup2}
	\sup_{\bm v\in\mathring{\mathcal S}_{r-1,r}^{+,1}}\frac{b_{1}(\bm v,q_h)}{\3bar\bm v\3bar_1}\geq \frac{b_{1}(\grad q_h,q_h)}{\3bar\grad q_h\3bar_1}=\|\grad q_h\|\geq C\|q_h\|_1.
\end{align}
Define 
\[N_h=\{\bm v_h\in \mathring{\mathcal S}_{r-1,r}^{+,1}:b_1(\bm v_h,q_h)=0, \ \forall q_h\in \mathring{\mathcal S}_{r}^0\}.\]
%Now we prove the coercivity. 
% For $\bm w_h\in \mathring{\mathcal S}_{r-1,r}^{+,1}$, we have $\curl \bm w_h\in \mathring{\mathcal S}_{r-1}^{+,2}$. 
%By \eqref{poincare}, we have
%\[\|\grad_h\curl \bm w_h\|\geq C \|\curl \bm w_h\|_{1,h}.\]
Following the argument in the Theorem \cite[Theorem 4.7]{hiptmair2002finite}, we can prove the following discrete Poincar\'e inequality for $ \bm w_h\in N_h$,
\[\|\curl \bm w_h\|\geq C\|\bm w
_h\|.\]
Then it holds the following coercivity
\begin{align}\label{coercivity}
	a_{1,h}(\bm w_h,\bm w_h)\geq C\3bar\bm w_h\3bar_{1}^2 \, \, \text{ for $\bm w_h\in N_h$}.
\end{align}
%Here $\|\bm v_h\|_{a_{1,h}}^2=a_1(\bm v_h,\bm v_h)$.

%The following Theorem implies the continuity condition that the spaces $\mathring{\mathcal S}_{r-2}^{3}$ should satisfy except for the above-mentioned three conditions.
\begin{theorem}\label{converg2}	Let $(\bm{u};p)\in H_0(\grad\curl;\Omega)\times H_0^1(\Omega)$ be the solution of the variational problem \eqref{prob2}. Assume that  $\mathring{\mathcal S}_{r-1,r}^{+,1}\subset H_0(\curl;\Omega)$ and $\mathring{\mathcal S}_{r}^0\subset H_0^1(\Omega)$ satisfy \eqref{bounda} -- \eqref{coercivity}. Then the problem \eqref{prob22} has a unique solution $(\bm u_h;p_h)$  such that $p_h=0$ and
	\begin{align}\label{estimateu}
		\3bar\bm u-\bm u_h\3bar_1&\leq C\Big(\inf_{\bm w_h\in \mathring{\mathcal S}_{r-1,r}^{+,1}}\3bar\bm u-\bm w_h\3bar_1+\sup_{\bm z_h\in N_h\backslash\{0\}}\frac{E_{1,h}(\bm u,\bm z_h)}{\3bar\bm z_h\3bar_1}\Big),
	\end{align}
	where the consistency error 
	\begin{align}\label{quadcurl:consistant}
		E_{1,h}(\bm u,\bm z_h)=-\mu \sum_{f\in \mathcal F_h}\langle \partial_{\bm n_f}(\curl\bm u),[\![\curl\bm z_h]\!]_f\rangle_f.
	\end{align}
	\end{theorem}
\begin{proof}
Taking $\bm v_h=\grad p_h$ in \eqref{prob22}, then 
\[(\grad p_h,\grad p_h)=0,\]
which together with the vanishing boundary condition of $p_h$ leads to $p_h=0$.

For any $\bm v_h\in \mathring{\mathcal S}_{r-1,r}^{+,1}$, by using integration by parts, we can show that 
\begin{align}\label{E2h}
 	a_{1,h}(\bm u,\bm v_h)+b_1(\bm v_h,p)-(\bm f,\bm v_h)=E_{1,h}(\bm u,\bm v_h).
 \end{align}
For any $\bm w_h\in N_h$, applying equations \eqref{prob22}, \eqref{bounda}, \eqref{coercivity}, and \eqref{E2h} yields
\begin{align*}
		&c\3bar\bm w_h-\bm u_h\3bar_1^2\leq a_{1,h}(\bm w_h-\bm u_h,\bm w_h-\bm u_h)\\
		=&a_{1,h}(\bm w_h-\bm u,\bm w_h-\bm u_h)+a_{1,h}(\bm u,\bm w_h-\bm u_h)-a_{1,h}(\bm u_h,\bm w_h-\bm u_h)\\
		=&a_{1,h}(\bm w_h-\bm u,\bm w_h-\bm u_h)+a_{1,h}(\bm u,\bm w_h-\bm u_h)+b_1(\bm w_h-\bm u_h,p_h)-(\bm f,\bm w_h-\bm u_h)\\
		=&a_{1,h}(\bm w_h-\bm u,\bm w_h-\bm u_h)+E_{1,h}(\bm u,\bm w_h-\bm u_h)+b_1(\bm w_h-\bm u_h,p_h-p)\\
		\leq &\3bar\bm w_h-\bm u\3bar_1\3bar\bm w_h-\bm u_h\3bar_1+E_{1,h}(\bm u,\bm w_h-\bm u_h)+b_1(\bm w_h-\bm u_h,p_h-p).
\end{align*}
The facts $p=0$ and $p_h=0$ lead to $b_1(\bm w_h-\bm u_h, p_h-p)=0$, and hence we have 
	\[\3bar\bm w_h-\bm u_h\3bar_1\leq \3bar\bm w_h-\bm u\3bar_1+\frac{E_{1,h}(\bm u,\bm w_h-\bm u_h)}{\3bar\bm w_h-\bm u_h\3bar_1}\leq \3bar\bm w_h-\bm u\3bar_1+\sup_{\bm z_h\in  N_h\backslash\{0\}}\frac{E_{1,h}(\bm u,\bm z_h)}{\3bar\bm z_h\3bar_1}.\]
Then the triangle inequality and \eqref{bounda} leads to
\begin{align}\label{Nh-estimate}
\3bar\bm u-\bm u_h\3bar_1\leq C \3bar\bm u-\bm w_h\3bar_1+C\sup_{\bm z_h\in N_h\backslash\{0\}}\frac{E_{1,h}(\bm u,\bm z_h)}{\3bar\bm z_h\3bar_1}.
\end{align}
Furthermore, for any $\bm v_h\in \mathring{\mathcal S}_{r-1,r}^{+,1}$, we decompose $\bm v_h$ as 
		\[\bm v_h=\grad q_h+\bm w_h\]
		with $\bm w_h\in N_h$ and $\grad q_h\in \grad \mathring{\mathcal S}_{r}^0$ satisfying 
		\[(\grad q_h,\grad \theta_h)=(\bm u-\bm v_h,\grad\theta_h) \text{ for any }\theta_h\in \mathring{\mathcal S}_{r}^0.\]
		 Taking $\theta_h=q_h$ yields $\|\grad q_h\|\leq \|\bm u-\bm v_h\|.$ Then
		\[\3bar\bm u-\bm w_h\3bar_1\leq\3bar\bm u-\bm v_h\3bar_1+\|\grad q_h\|\leq 2\3bar\bm u-\bm v_h\3bar_1,\]
		which, combined with \eqref{Nh-estimate}, leads to
		\[\3bar\bm u-\bm u_h\3bar_1\leq C \3bar\bm u-\bm v_h\3bar_1+C\sup_{\bm z_h\in N_h\backslash\{0\}}\frac{E_{1,h}(\bm u,\bm z_h)}{\3bar\bm z_h\3bar_1}.\]
			Taking infimum over $\bm v_h\in \mathring{\mathcal S}_{r-1,r}^{+,1}$ on both sides leads to \eqref{estimateu}.
%		\begin{align*}
%		\|\bm u-\bm u_h\3bar_1\leq C\inf_{\bm v_h\in \mathring{\mathcal S}_{r-1,r}^{+,1}}\|\bm u-\bm v_h\3bar_1+C\sup_{\bm z_h\in N_h\backslash\{0\}}\frac{E_{1,h}(\bm u,\bm z_h)}{\|\bm z_h\|_{a_{1,h}}}.
%		\end{align*}

\end{proof}
\begin{theorem}\label{Thm:approx-quadcurl}
Let $(\bm{u};p)\in H_0(\grad\curl;\Omega)\times H_0^1(\Omega)$ be the solution of the variational problem \eqref{prob2} and $(\bm{u}_{h};p_{h})\in  \mathring{\mathcal S}_{r-1,r}^{+,1}\times \mathring{\mathcal S}_{r}^0$ be the solution of the discrete problem \eqref{prob22}. Suppose that $\bm{u} \in H^{s}(\curl;\Omega)$ with $2\leq s\leq r$, then there holds  $p=p_h=0$ and 
\begin{equation}\label{err:u:anorm22}
\3bar\bm u-\bm u_h\3bar_1\leq Ch^{s-1}\big(h|\bm{u}|_{s}+h|\curl\bm{u}|_{s}+\mu^{1/2}|\curl\bm u|_{s}\big).
\end{equation}
\end{theorem}
\begin{proof}
	In light of Theorem \ref{est-pi1}, for $\bm{u}\in  H^{s}(\curl;\Omega)$ with $2\leq s\leq r$,  there holds
\begin{equation}\label{interpolation:pi2h:a1h22}
\begin{split}
\3bar\bm{u}-\bm\Pi^1_h\bm{u}\3bar_1&\leq C \left(\mu^{1/2}|\curl(\bm{u}-\bm\Pi^1_h\bm{u})|_{1,h}+\|\curl(\bm{u}-\bm\Pi^1_h\bm{u})\|+\|\bm{u}-\bm\Pi^1_h\bm{u}\|\right)\\
&\leq C \left(\mu^{1/2}h^{s-1}|\curl\bm{u}|_{s}+h^{s}|\curl\bm{u}|_{s}+h^{s}|\bm{u}|_{s}\right). 
\end{split}
\end{equation}
%can be obtained by \cite[Theorem 3.3.3]{Wang-Shi-FEMbook}. 

To estimate the consistency error \eqref{quadcurl:consistant}, define $\mathcal P_f^k:[L^2(f)]^3\rightarrow [P_k(f)]^3$ to be the $L^2$ orthogonal projection and let $\bm{w}_{h}\in  N_h$. 
%Then we have
%\begin{equation}\label{continuitycondition:t}
%\langle[\![\bm{w}_{h}\times \bm{n}]\!]_f, \bm{q}\rangle_{f}=0\quad\text{for all}\, \, \bm{q}\in [P_{r-2}(f)]^{2} \text{ and } f\in\mathcal{F}_{h}.
%\end{equation}
Noting that $\curl\bm{w}_{h}\in \curl \mathring{\mathcal S}_{r-1,r}^{+,1}\subseteq \mathcal S_{r-1}^{+,2}(\mathcal T_h)$ and using the definition of $\mathcal S_{r-1}^{+,2}(\mathcal T_h)$, we have
%\begin{equation}
%\begin{split}
%|E_{1,h}(\bm u,\bm w_{h})|&=\Big|\sum_{f\in \mathcal F_h}\langle \nu\curl\bm u,[\![\bm w_{h}\times \bm n_f]\!]_f\rangle_f\Big|\\
%&=\Big|\sum_{f\in \mathcal F_h}\nu\langle \curl\bm u-\mathcal{P}_{f}^{r-2}(\curl\bm u),[\![\bm w_{h}\times \bm n_f- \mathcal{P}_{f}^{0}(\bm w_{h}\times \bm n_f)]\!]_f\rangle_f\Big|\\
%\end{split}
%\end{equation}
%where $\mathcal{P}_{f}^{r-3}$ and $\mathcal{P}_{f}^{0}$ are the orthogonal projection operators from $[L^{2}(f)]^{3}$ to $[P_{r-3}(f)]^{3}$ and $[P_{0}(f)]^{3}$, respectively. The error estimate for these orthogonal projection operators \cite[Theorem 3.3.5]{Wang-Shi-FEMbook} and Schwarz inequality lead to 
\begin{equation}
\label{consistent:err:22}
\begin{split}
E_{1,h}(\bm u,\bm w_h)&=-\mu\sum_{f\in \mathcal F_h}\langle \partial_{\bm n_f}(\curl\bm u),[\![\curl\bm w_h]\!]_f\rangle_f \\
&=-\mu\sum_{f\in \mathcal F_h}\langle \partial_{\bm n_f}(\curl\bm u)-\mathcal P_{f}^{r-2}\partial_{\bm n_f}(\curl\bm u),[\![\curl\bm w_h]\!]_f\rangle_f\\
&=-\mu\sum_{K\in \mathcal T_h}\sum_{f\in \mathcal F_h(K)}\langle \partial_{\bm n_{f,K}}(\curl\bm u)-\mathcal P_{f}^{r-2}\partial_{\bm n_{f,K}}(\curl\bm u),\curl\bm w_h\rangle_{f}\\
&=-\mu\sum_{K\in \mathcal T_h}\sum_{f\in \mathcal F_h(K)}\langle \partial_{\bm n_{f,K}}(\curl\bm u)-\mathcal P_{f}^{r-2}\partial_{\bm n_{f,K}}(\curl\bm u),(\curl\bm w_h-\mathcal P_{f}^{0}\curl\bm w_h)\rangle_{f}\\
&\leq C\mu^{1/2} h^{s-1}|\curl\bm u|_{s}\3bar\bm{w}_{h}\3bar_1.
%|E_{1,h}(\bm u,\bm w_{h})|
%&\leq C \nu^{1/2} h^{s-1}|\bm u|_{s}\|\bm{w}_{h}\|_{a_{1,h}}.
\end{split}
\end{equation}
According to Theorem \ref{converg2}, we can arrive at \eqref{err:u:anorm22} by combining \eqref{interpolation:pi2h:a1h22} and  \eqref{consistent:err:22}.

\end{proof}
\begin{remark}\label{high-conv-quadcurl}
When $\mu$ is sufficiently small, the first two terms $h|\bm{u}|_{s}+h|\curl\bm{u}|_{s}$ in the right hand side of \eqref{err:u:anorm22} will dominate, and hence $\3bar\bm u-\bm u_h\3bar_1$ will have one-order higher accuracy.
\end{remark}

\section{Applications to Brinkman Problem}
In this section, we apply $\mathring{\mathcal S}_{r-1}^{+,2}(\mathcal T_h)$ and $\mathring{\mathcal S}_{r-2}^{3}(\mathcal T_h)$ to the Brinkman problem \eqref{prob2}. To simplify notation, we drop $\mathcal T_h$ and simply denote $\mathring{\mathcal S}_{r-1}^{+,2}$ and $\mathring{\mathcal S}_{r-2}^{3}$.

The variational formulation is to find $(\bm u;p)\in [ H_0^1(\Omega)]^3\times L_0^2(\Omega)$  such that
\begin{equation}\label{variation-prob-brinkman}
\begin{split}
a_2(\bm u,\bm v)-b_2(\bm v,p)&=(\bm f, \bm v)\quad \forall \bm v\in [H_0^1(\Omega)]^3,\\
b_2(\bm u,w)&=(g,w)\quad \forall w\in L_0^2(\Omega),
\end{split}
\end{equation}
where $a_2(\bm u,\bm v):=\nu(\grad\bm u, \grad\bm v) + \alpha(\bm u,\bm v)$ and $b_2(\bm v,p)=(\div \bm v,p)$.

% We take
% \begin{equation*}
% \begin{split}
% \bm{V}_{h}=\mathring{\mathcal S}_{r-1}^{+,2}(\mathcal T_h)&=\{ \bm{v}\in  H_{0}(\operatorname{div};\Omega):\, \bm{v}|_{K}\in \mathcal S_{r-1}^{+,2}(K)\ \text{for all}~K\in\mathcal{T}_{h},  \\
% &\langle [\![\bm{v}\times \bm{n}_f]\!]_f, \bm{q}\rangle_{f}=0\ \text{for all}~\bm{q}\in [P_{r-2}(f)]^{2}\text{ and } f\in\mathcal{F}_{h}\},\\
% \mathring{\mathcal S}_{r-2}^{3}=\mathring{\mathcal S}_{r-2}^{3}(\mathcal T_h)&=\{ w\in  L_0^2(\Omega):\, w|_{K}\in \mathcal S_{r-2}^{3}(K)\ \text{for all}~K\in\mathcal{T}_{h}\},
% \end{split}
% \end{equation*}
% for $r\geq 2$.

The nonconforming finite element method seeks $(\bm u_h;p_h)\in \mathring{\mathcal S}_{r-1}^{+,2}\times \mathring{\mathcal S}_{r-2}^{3}$ such that 
\begin{equation}\label{FEM-brinkman}
\begin{split}
a_{2,h}(\bm u_h,\bm v_h)-b_{2}(\bm v_h,p_h)&=(\bm f, \bm v_h)\quad \forall \bm v_h\in \mathring{\mathcal S}_{r-1}^{+,2},\\
b_{2}(\bm u_h,w_h)&=(g,w_h)\quad \forall w_h\in \mathring{\mathcal S}_{r-2}^{3},
\end{split}
\end{equation}
where $a_{2,h}(\bm u_h,\bm v_h):=\sum\limits_{K\in\mathcal T_h}\nu(\grad\bm u_h, \grad\bm v_h)_K + \alpha(\bm u_h,\bm v_h)$.

For $\bm v\in \mathring{\mathcal S}_{r-1}^{+,2}+[H_0^1(\Omega)]^3$, we define 
\[\3bar\bm v\3bar_2:=\left(a_{2,h}(\bm v,\bm v)+M\|\div\bm v\|^2\right)^{\frac{1}{2}} \text{ with }M=\max\{\nu,\alpha\}.\]
To show that scheme \eqref{FEM-brinkman} is uniformly stable with respect to $\nu$, we first have
\begin{align*}
&a_{2,h}(\bm u_h,\bm v_h)\leq \3bar\bm u_h\3bar_2\3bar\bm v_h\3bar_2\, &&\text{for }\bm u_h,\bm v_h\in \mathring{\mathcal S}_{r-1}^{+,2}+[H_0^1(\Omega)]^3,\\
&b_{2}(\bm v_h,q_h)\leq C\3bar\bm v_h\3bar_2\|q_h\|\, &&\text{for }\bm u_h\in \mathring{\mathcal S}_{r-1}^{+,2}+[H_0^1(\Omega)]^3\text{ and }q_h\in L_0^2(\Omega).
\end{align*}

We then show that the scheme is inf-sup stable. 
\begin{lemma}[Inf-sup Condition]
For any $w_h\in \mathring{\mathcal S}_{r-2}^{3}$, there exists a constant $c>0$ such that 
\begin{align}\label{infsup}
	\sup_{\bm v_h\in\mathring{\mathcal S}_{r-1}^{+,2}}\frac{b_{2}(\bm v_h,w_h)}{\3bar\bm v_h\3bar_2}\geq c\|w_h\|.
\end{align}
%where $\|\cdot\|_{2,h}^2=\sum_{K\in\mathcal T_h}\|\cdot\|_{2,K}^2.$
\end{lemma}
\begin{proof}
	From \cite[Corollary 2.4]{Girault2012Finite}, for any $w_h\in \mathring{\mathcal S}_{r-2}^{3}\subset L_0^2(\Omega)$, there exists $\bm v\in  [H_0^1(\Omega)]^3$ satisfying
\begin{align*}
%	C\|q\|\leq \frac{b_{2,h}(\bm v,q)}{\|\bm v\|_2}.
\div\bm v= w_h\text{ \quad and \quad } \|\bm v\|_1\leq C\|w_h\|. 
\end{align*}
By \eqref{commute:a2}, it holds
\[\div \bm \Pi_h^2\bm v = \bm\pi_h^3\div\bm v = w_h.\]
In addition, Theorem \ref{err:interpolation:2h} leads to
\begin{equation*}
|\bm\Pi^2_h\bm{v}|_{m,K}\leq C_0\|\bm{v}\|_{1,K}\quad\text{ with $m=0,1$}.
\end{equation*}
Therefore, for any $w_h\in \mathring{\mathcal S}_{r-2}^{3}$,
\[\sup_{\bm v_h\in\mathring{\mathcal S}_{r-1}^{+,2}}\frac{b_{2}(\bm v_h,w_h)}{\3bar\bm v_h\3bar_2}\geq\frac{b_{2}(\bm\Pi^2_h\bm v,w_h)}{\3bar\bm\Pi^2_h\bm v\3bar_2}\geq \frac{\|w_h\|^2}{\sqrt{M}\big(\sqrt{2}C_0\|\bm v\|_{1}+\|w_h\|\big)}\geq
c\|w_h\|\]
with $c=\frac{1}{\sqrt{M}(1+\sqrt{2}C_0)}$.
\end{proof}

In addition to the inf-sup condition, the following lemma presents an additional condition to 
make the numerical scheme uniformly stable with respect to $\nu$. 

\begin{lemma}\label{div-free}
Define
	\[Z_h:=\{\bm v\in\mathring{\mathcal S}_{r-1}^{+,2}:b_{2}(\bm v,w)=0\ \ \forall w\in \mathring{\mathcal S}_{r-2}^{3}\}.\] Then we have
\begin{align}\label{div-free2}
	Z_h =\{\bm v\in\mathring{\mathcal S}_{r-1}^{+,2}:\div\bm v =0\}. 
\end{align}
\end{lemma}
\begin{proof}
	The lemma is a direct result of Theorem \ref{exactness2}. 
\end{proof}

With Lemma \ref{div-free}, it is clear that 
\[a_{2,h}(\bm v_h,\bm v_h)=\3bar\bm v_h\3bar_2^2\quad \text{for }\bm v_h\in Z_h.\]

According to \cite[Theorem 3.1]{xie2008uniformly}, 
the finite element scheme \eqref{FEM-brinkman} is uniformly stable with respect to $\nu$ under the norm $\3bar\cdot\3bar_2$.

%\noindent The condition \eqref{div-free} holds if the following stronger condition is satisfied
%\begin{align}\label{div-free2}
%\div \mathring{\mathcal S}_{r-1}^{+,2}\subseteq \mathring{\mathcal S}_{r-2}^{3}.
%\end{align}

%According to Theorem \ref{stalbecond}, the $\div$-conforming finite element presented in Section \ref{derhamFEM} is a good candidate for the space $\mathring{\mathcal S}_{r-1}^{+,2}$ since it satisfies both the conditions \eqref{infsup} and \eqref{div-free2} and it has fewer DOFs compared to the  RT elements. However, to make the numerical scheme converge, we need some tangential continuity across interior faces according to the following theorem.

We will use the following standard result to show the convergence of scheme \eqref{FEM-brinkman}.
\begin{proposition}\label{converg}(\cite[Theorem 3.2]{xie2008uniformly})
Let $(\bm{u};p)\in [H^{1}_{0}(\Omega)]^{3}\times L^{2}(\Omega)$ be the solution of the variational problem \eqref{variation-prob-brinkman}.
	Assume that  $\mathring{\mathcal S}_{r-1}^{+,2}\subset H_0(\div;\Omega)$ satisfy \eqref{infsup} and \eqref{div-free2}. Then the problem \eqref{FEM-brinkman} has a unique solution $(\bm u_h,p_h)$  such that
	\begin{align*}
		\3bar\bm u-\bm u_h\3bar_2&\leq C\Big(\inf_{\bm v\in \mathring{\mathcal S}_{r-1}^{+,2}}\3bar\bm u-\bm v\3bar_2+\sup_{\bm w\in \mathring{\mathcal S}_{r-1}^{+,2}\backslash\{0\}}\frac{E_{2,h}(\bm u,\bm w)}{\3bar\bm w\3bar_2}\Big),\\
		\|p-p_h\|&\leq C\Big(\inf_{w\in \mathring{\mathcal S}_{r-2}^{3}}\|p-w\|+\inf_{\bm v\in \mathring{\mathcal S}_{r-1}^{+,2}}\3bar\bm u-\bm v\3bar_2+\sup_{\bm w\in \mathring{\mathcal S}_{r-1}^{+,2}\backslash\{0\}}\frac{E_{2,h}(\bm u,\bm w)}{\3bar\bm w\3bar_2}\Big),
	\end{align*}
	where the consistency error 
	\begin{equation}\label{brinkman:consistant}
	E_{2,h}(\bm u,\bm v)=\sum_{f\in \mathcal F_h}\langle \nu\curl\bm u,[\![\bm v\times \bm n_f]\!]_f\rangle_f.
	\end{equation}
\end{proposition}

\begin{theorem}\label{convergence-num-sol}
Let $(\bm{u};p)\in [H^{1}_{0}(\Omega)]^{3}\times L^{2}(\Omega)$ be the solution of the variational problem \eqref{variation-prob-brinkman} and $(\bm{u}_{h};p_{h})\in \mathring{\mathcal S}_{r-1}^{+,2}\times  \mathring{\mathcal S}_{r-2}^{3}$ be the solution of the discrete problem \eqref{FEM-brinkman}. Suppose that $\bm{u} \in [H^{s}(\Omega)]^{3}$ with $2\leq s\leq r$, then there holds
\begin{equation}\label{err:u:anorm}
\3bar\bm{u}-\bm{u}_{h}\3bar_2\leq C(\nu^{1/2}h^{s-1}+\alpha^{1/2}h^{s})|\bm{u}|_{s}.
\end{equation}
For $p\in H^{s-1}(\Omega)$, there holds
\begin{equation}\label{err:p:L2norm}
\|p-p_{h}\|\leq C\big(h^{s-1}|p|_{s-1}+(\nu^{1/2}h^{s-1}+\alpha^{1/2}h^{s})|\bm{u}|_{s}\big).
\end{equation}
\end{theorem}
\begin{proof}
%We will apply Theorem \ref{converg} to prove \eqref{err:u:anorm} and \eqref{err:p:L2norm} since the conditions of Theorem \ref{converg} are satisfied. 
%
In light of Theorem \ref{err:interpolation:2h}, for $\bm{u}\in  [H^{s}(\Omega)]^{3}$ with $2\leq s\leq r$,  there holds
\begin{equation}\label{interpolation:pi2h:a1h}
\begin{split}
\3bar\bm{u}-\bm\Pi^2_h\bm{u}\3bar_2&\leq C \left(\nu^{1/2}|\bm{u}-\bm\Pi^2_h\bm{u}|_{{1},h}+\alpha^{1/2}\|\bm{u}-\bm\Pi^2_h\bm{u}\|+M^{1/2}\|\div\bm{u}-\div\bm\Pi^2_h\bm{u}\|\right)\\
&\leq C \left(\nu^{1/2}h^{s-1}+\alpha^{1/2}h^{s}+M^{1/2}h^{s-1}\right)|\bm{u}|_{s}. 
\end{split}
\end{equation}
For $p \in  H^{s-1}(\Omega)$, there holds the following the estimation:
\begin{equation}\label{projection:wh:err}
\|p-\bm\pi^3_{h}p\|\leq C h^{s-1}|p|_{s-1}.
\end{equation}
%can be obtained by \cite[Theorem 3.3.3]{Wang-Shi-FEMbook}. 

To estimate the consistency error \eqref{brinkman:consistant},
let $\bm{w}_{h}\in \mathring{\mathcal S}_{r-1}^{+,2}$. 
%Then we have
%\begin{equation}\label{continuitycondition:t}
%\langle[\![\bm{w}_{h}\times \bm{n}]\!]_f, \bm{q}\rangle_{f}=0\quad\text{for all}\, \, \bm{q}\in [P_{r-2}(f)]^{2} \text{ and } f\in\mathcal{F}_{h}.
%\end{equation}
Proceeding as in \eqref{consistent:err:22}, the consistent error has the following estimate
%\begin{equation}
%\begin{split}
%|E_{2,h}(\bm u,\bm w_{h})|&=\Big|\sum_{f\in \mathcal F_h}\langle \nu\curl\bm u,[\![\bm w_{h}\times \bm n_f]\!]_f\rangle_f\Big|\\
%&=\Big|\sum_{f\in \mathcal F_h}\nu\langle \curl\bm u-\mathcal{P}_{f}^{r-2}(\curl\bm u),[\![\bm w_{h}\times \bm n_f- \mathcal{P}_{f}^{0}(\bm w_{h}\times \bm n_f)]\!]_f\rangle_f\Big|\\
%\end{split}
%\end{equation}
%where $\mathcal{P}_{f}^{r-3}$ and $\mathcal{P}_{f}^{0}$ are the orthogonal projection operators from $[L^{2}(f)]^{3}$ to $[P_{r-3}(f)]^{3}$ and $[P_{0}(f)]^{3}$, respectively. The error estimate for these orthogonal projection operators \cite[Theorem 3.3.5]{Wang-Shi-FEMbook} and Schwarz inequality lead to 
\begin{equation}
\label{consistent:err:1}
\begin{split}
&E_{2,h}(\bm u,\bm w_h)=\sum_{f\in \mathcal F_h}\langle \nu\curl\bm u,[\![\bm w_h\times \bm n_f]\!]_f\rangle_f \\
%&=\nu\sum_{f\in \mathcal F_h}\langle \curl\bm u-\mathcal P_{f}^{r-2}\curl\bm u,[\![\bm w_h\times \bm n_f]\!]_f\rangle_f\\
%&=\nu\sum_{K\in \mathcal T_h}\langle \curl\bm u-\mathcal P_{f}^{r-2}\curl\bm u,\bm w_h\times \bm n_K\rangle_{\partial K}\\
=\ &\nu\sum_{K\in \mathcal T_h}\sum_{f\in \mathcal F_h(K)}\langle \curl\bm u-\mathcal P_{f}^{r-2}\curl\bm u,(\bm w_h-\mathcal P_{f}^{0}\bm w_h)\times \bm n_{f,K}\rangle_{f}\\
\leq\  &C\nu^{1/2} h^{s-1}|\bm u|_{s}\3bar\bm{w}_{h}\3bar_2.
%|E_{2,h}(\bm u,\bm w_{h})|
%&\leq C \nu^{1/2} h^{s-1}|\bm u|_{s}\|\bm{w}_{h}\|\gc,h_{2,h}}.
\end{split}
\end{equation}
According to Proposition \ref{converg}, we can arrive at \eqref{err:u:anorm} and \eqref{err:p:L2norm}
by combining \eqref{interpolation:pi2h:a1h}, \eqref{projection:wh:err}, and \eqref{consistent:err:1}.
\end{proof}
\begin{remark}\label{high-conv-Brinkman}
When $\nu$ is sufficiently small, the second term $\alpha^{1/2}h^{s}|\bm{u}|_{s}$ in the right hand side of \eqref{err:u:anorm} will dominate, and consequently, $\3bar\bm u-\bm u_h\3bar_2$ will have one-order higher accuracy.
\end{remark}

\begin{remark}
From the definition of $\mathring{\mathcal S}_{r-1}^{+,2}$, $\int_f [\![\bm v_h]\!]\d S=0$ for all $f\in \mathcal F_h$ and $\bm v_h\in \mathring{\mathcal S}_{r-1}^{+,2}$.
By \cite[(1.3)]{brenner2003poincare}, we have
$|\bm v_h|_{1,h}\geq C \|\bm v_h\|_{1,h}$,
where  $\|\cdot\|_{1,h}=\sum\limits_{K\in\mathcal T_h}\|\cdot\|_{1,K}$. Then the nonconforming element can also be used to solve Stokes problems with $\alpha = 0$ in \eqref{prob-brinkman}.
\end{remark}

\section{Numerical experiments}
\begin{example}
\end{example}
We use the finite element spaces $\mathring{\mathcal S}_{r-1,r}^{+,1}$ and $\mathring{\mathcal S}_{r}^0$, $r=2,3$, to solve the $-\curl\Delta\curl$ problem with $\gamma =1 $ and $\mu=1$, $10^{-2}$, $10^{-4}$, and $10^{-6}$ on $\Omega =[0,1]\times[0,1]\times[0,1]$. We use the exact solution
\begin{align*}
\bm u =&\left(
\begin{array}{c}
\sin^3(\pi x)\sin^2(\pi y)\sin^2(\pi z)\cos(\pi y)\cos(\pi z) \\
\sin^3(\pi y)\sin^2(\pi z)\sin^2(\pi x)\cos(\pi z)\cos(\pi x)\\
-2\sin^3(\pi z)\sin^2(\pi x)\sin^2(\pi y)\cos(\pi x)\cos(\pi y)
\end{array}
\right).
\end{align*}
Then the right-hand side can be derived by a simple calculation. In Figure \ref{fig:rate2}, we plot the errors $\bm u-\bm u_h$ in the sense of different norms. 
As evident from Figure \ref{fig:rate2}, the errors converge to 0 with order $r-1$ for different values of $\mu$. It can also be observed in Figure \ref{gcerror} that the convergence rate of $\3bar\bm{u}-\bm{u}_h\3bar_1$ approaches to $r$ when $\mu=10^{-6}$, which is consistent with the observation in Remark \ref{high-conv-quadcurl}.

%\begin{table}[h]
%	\centering
%	\caption{Numerical results by using the lowest-order   $H(\grad\curl)$ nonconforming element} \label{tab1}
%	\begin{tabular}{cccccccc}
%		\hline
%		$h$ &$\left\|\bm e_h\right\|$&rates&$\left\|\nabla\times\bm e_h\right\|$&rates&$\left\|(\nabla\times)^2\bm e_h\right\|$& rates\\
%		\hline
%		$1\slash 8$ & 4.4096438107044e-03&2.5193& 7.7902828669093e-02&2.2247&3.3025950534052e+00 
%		&0.9968  \\
%		$1\slash 10$&2.5133750608893e-03&2.3265&4.7419418073851e-02 &2.1350&2.6439856815847e+00 		&0.9899 \\
%		$1\slash 12$&1.6445364076257e-03&-&3.2129238459824e-02&-&2.2073965219030e+00
%&\\
%			\hline
%	\end{tabular}
%\end{table}
\begin{figure}[h!]
\captionsetup[subfigure]{font=footnotesize}
\centering
\subfloat[The convergence rate of $\|\bm{u}-\bm{u}_h\|$.]{\label{uherror} \includegraphics[width=0.5\textwidth]{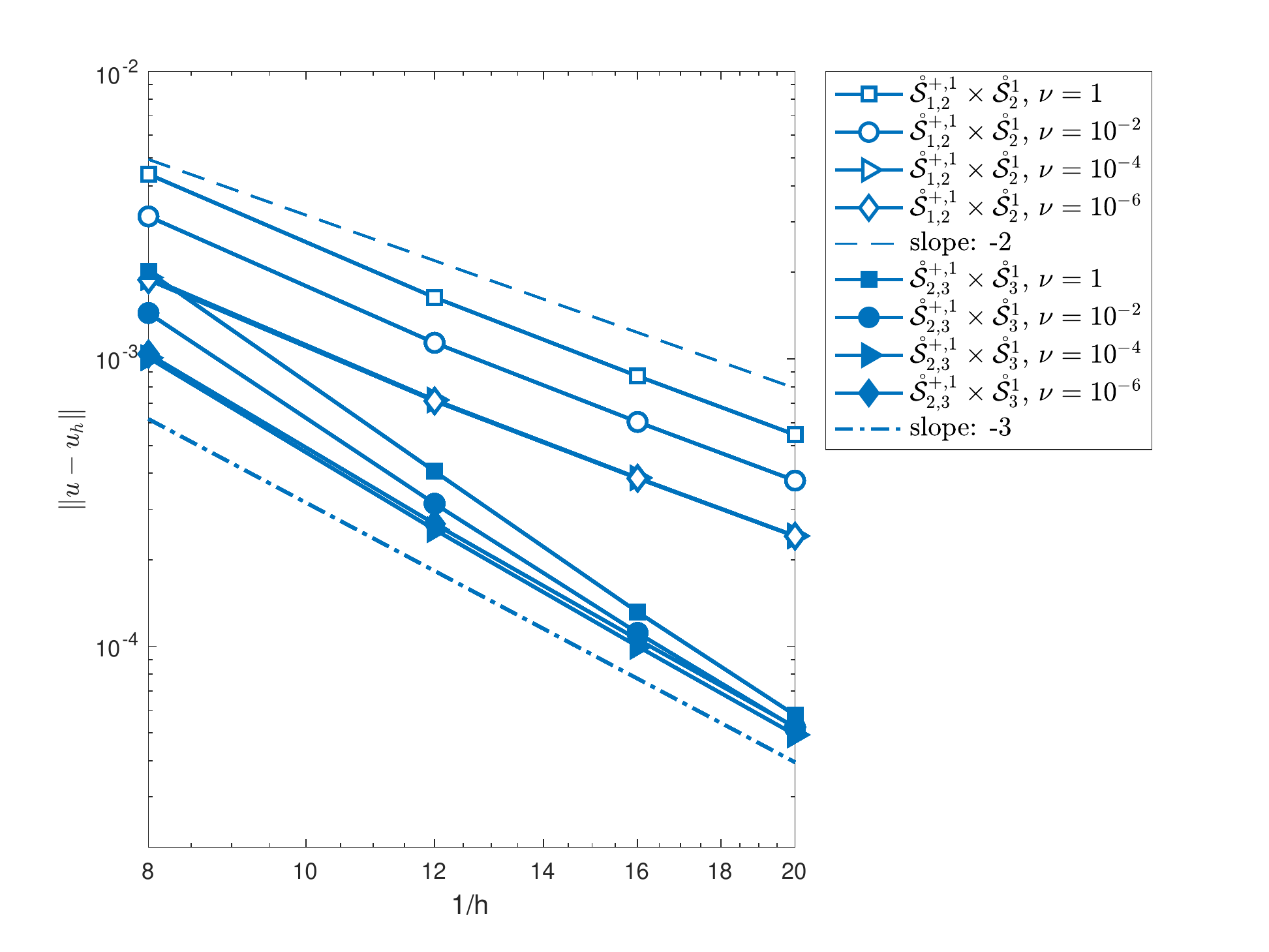}}%请勿换行
\subfloat[The convergence rate of $\|\curl\bm u -\curl\bm u_h\|$.]{\label{curluherror}\includegraphics[width=0.5\textwidth]{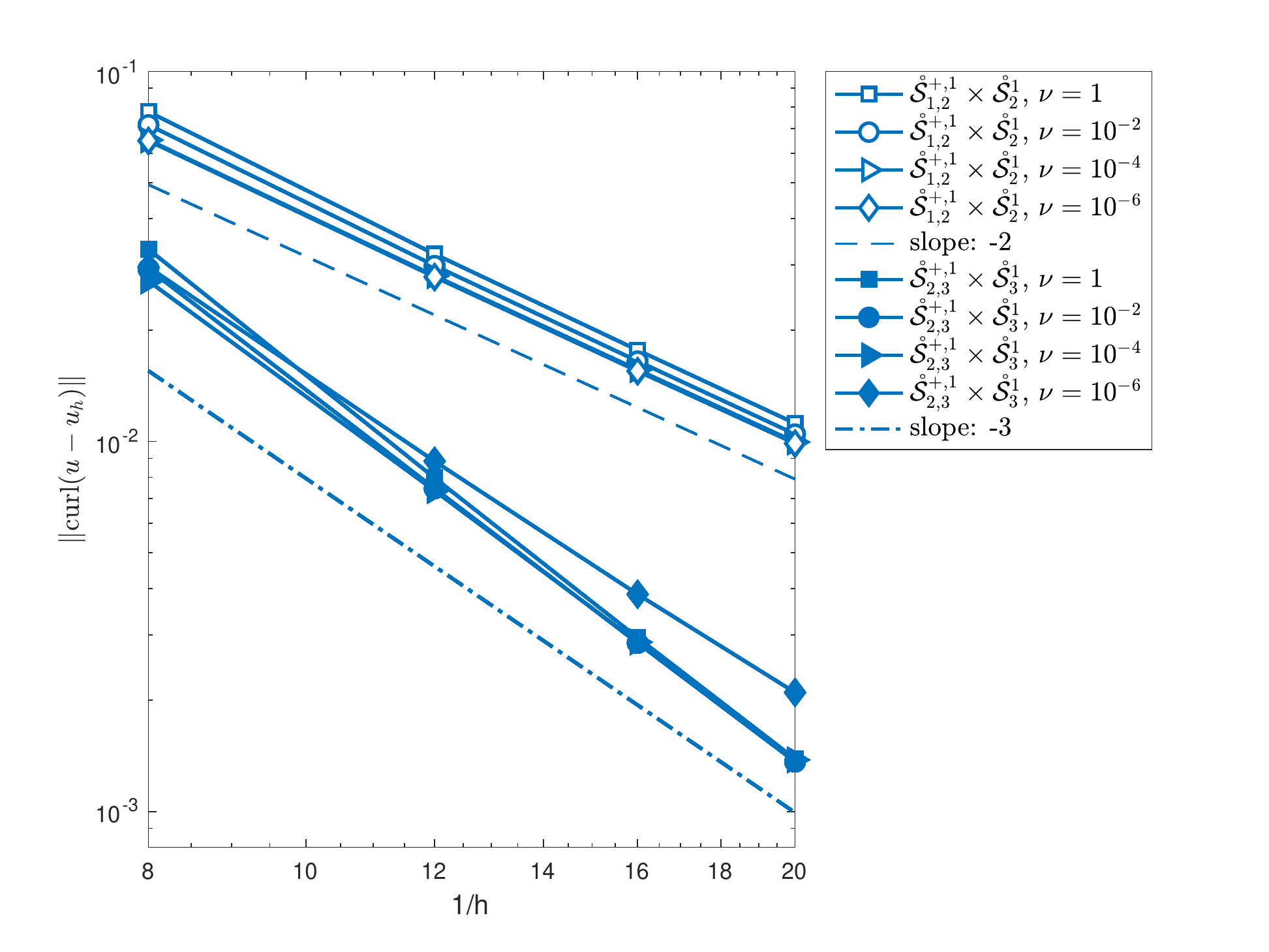}}\\
\subfloat[The convergence rate of $|\curl\bm{u}-\curl\bm{u}_h|_{1,h}$.]{\label{curlcurluherror}\includegraphics[width=0.5\textwidth]{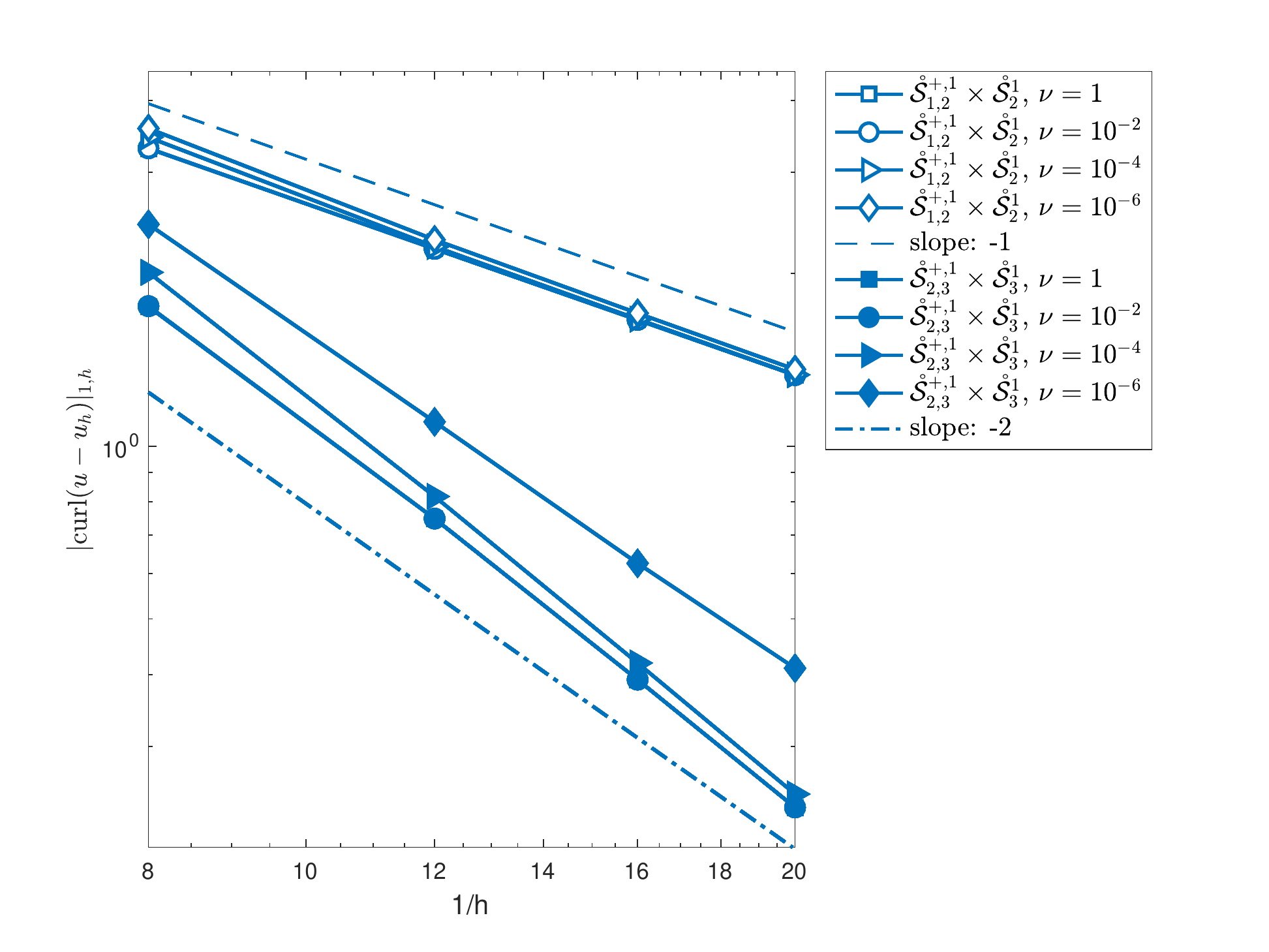}}
\subfloat[The convergence rate of $\3bar\bm{u}-\bm{u}_h\3bar_1$.]{\label{gcerror}\includegraphics[width=0.5\textwidth]{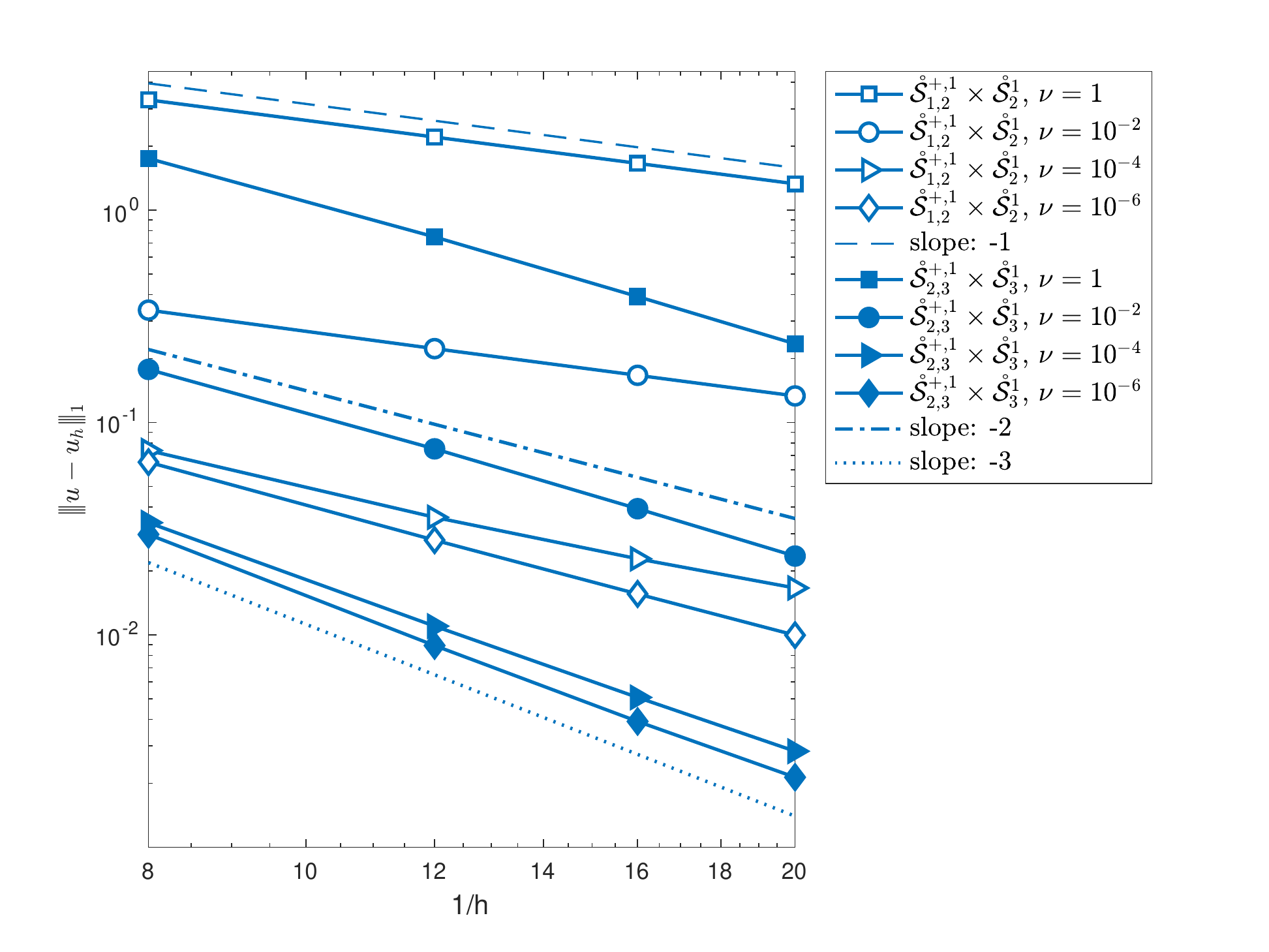}}
\caption{Convergence rates for Example 1}
\label{fig:rate2}
\end{figure}

\begin{example}
\end{example}
We consider the Brinkman problem \eqref{prob-brinkman} with $\Omega =[0,1]\times[0,1]\times[0,1]$.
Cut the unit cube into eight cubes and each of the eight cubes is refined into eight half-sized
cubes uniformly to get a higher-level grid. The mesh size $h$ of each level is $1/2$, $1/4$, $1/8$, $1/16$, respectively.

We use the finite element space $\mathring{\mathcal S}_{r-1}^{+,2}\times \mathring{\mathcal S}_{r-2}^{3}$ with $r=2,3$ to discretize $(\bm{u};p)$ in \eqref{variation-prob-brinkman}. Let $\alpha = 1$ in \eqref{prob-brinkman}, and let the effective viscosity $\nu$ be $1$, $10^{-2}$, $10^{-4}$, and $10^{-6}$. We consider the system \eqref{prob-brinkman} with $g=0$ and
$\bm f=-\div(\nu \grad\bm u)+\alpha\bm u+\grad p$, where 
\begin{equation}
\bm{u}=\operatorname{curl}\begin{pmatrix}
y^{2}(1-y)^{2}x(1-x)z^{2}(1-z)^{3}\\
x^{2}(1-x)^{2}y(1-y)z^{2}(1-z)^{3}\\
0
\end{pmatrix},
\end{equation}
and 
\begin{equation}
p = \left(x-\frac{1}{2}\right)\left(y-\frac{1}{2}\right)(1-z).
\end{equation}

For different values of $\nu$ and $h$, we present the errors of the velocity $\bm{u}_h$ measured in various norms and the errors of the pressure $p_h$ measured in the $L^{2}$ norm in Figure \ref{fig:rate}.
It can be observed that the optimal order of convergence is achieved, and the discretization is stable as $\nu$ tends to 0. In particular, Figure \ref{converfig:d} shows that $\3bar\bm u-\bm u_h\3bar_2$ has one-order higher convergence rate when $\nu$ is sufficiently small, which is in accordance with Remark \ref{high-conv-Brinkman}.

%When $\mu$ is sufficiently small, the first two terms $h|\bm{u}|_{s}+h|\curl\bm{u}|_{s}$ in the right hand side of \eqref{err:u:anorm22} will dominate, and hence $\3bar\bm u-\bm u_h\3bar_1$ will have one-order higher accuracy.

\begin{figure}[h]
\captionsetup[subfigure]{font=footnotesize}
\centering
\subfloat[The convergence rates of $|\bm{u}-\bm{u}_h|_{1,h}$.]{\label{converfig:a} \includegraphics[width=0.5\textwidth]{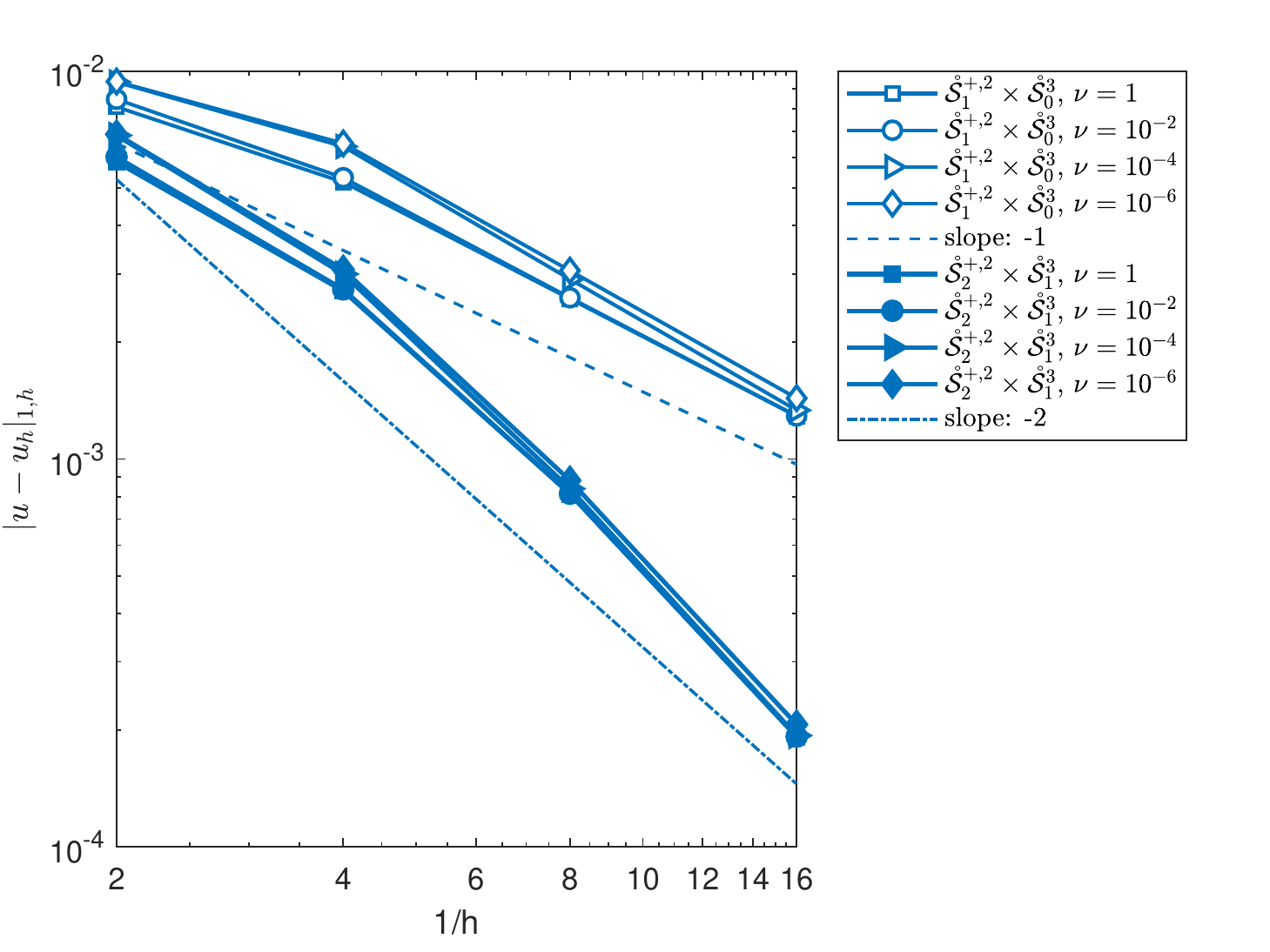}}%请勿换行
\subfloat[The convergence rate of $\|\bm{u}-\bm{u}_h\|$.]{\label{converfig:c}\includegraphics[width=0.5\textwidth]{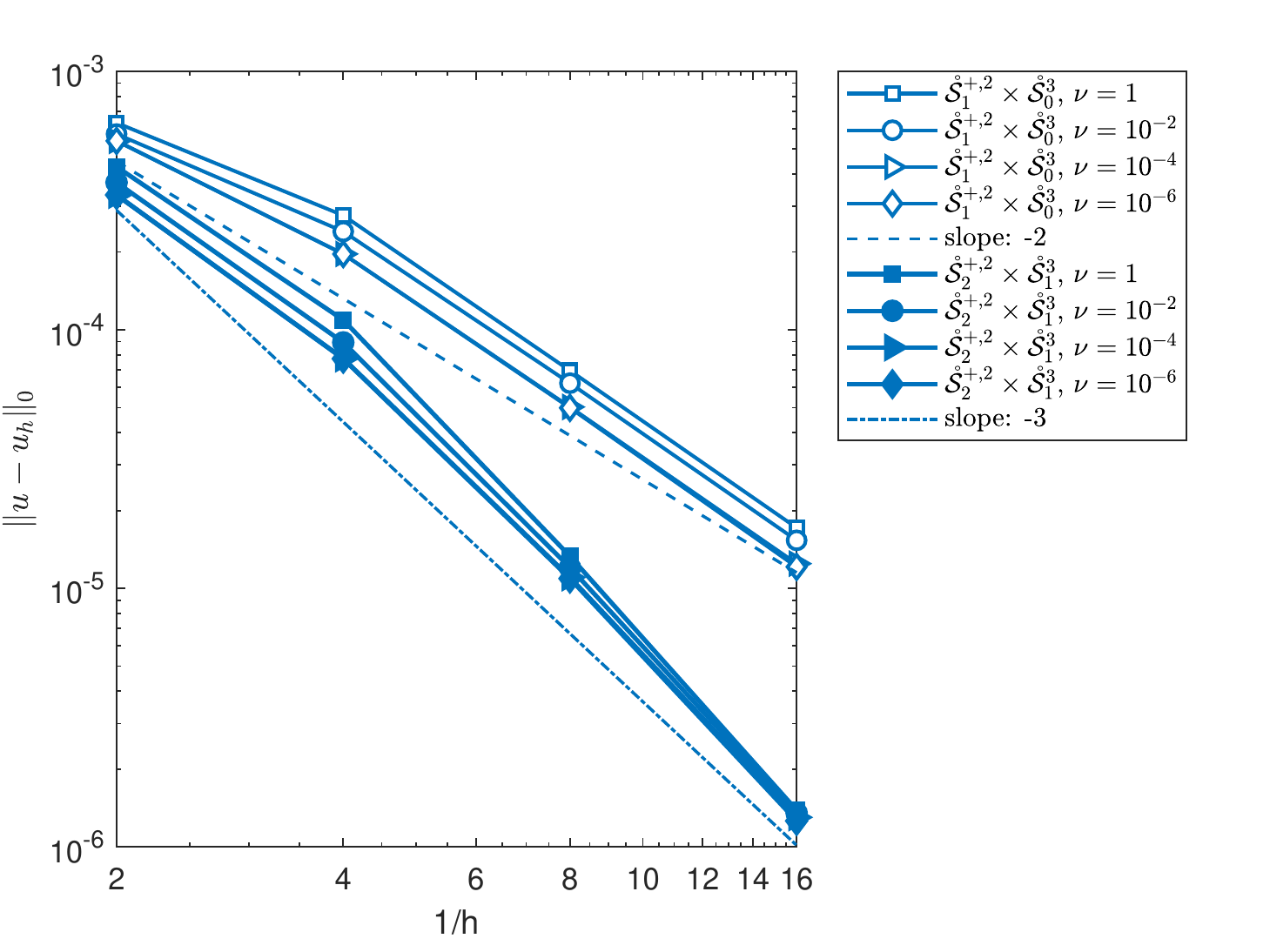}}\\
\subfloat[The convergence rate of $\3bar\bm{u}-\bm{u}_{h}\3bar_2$.]{\label{converfig:d}\includegraphics[width=0.5\textwidth]{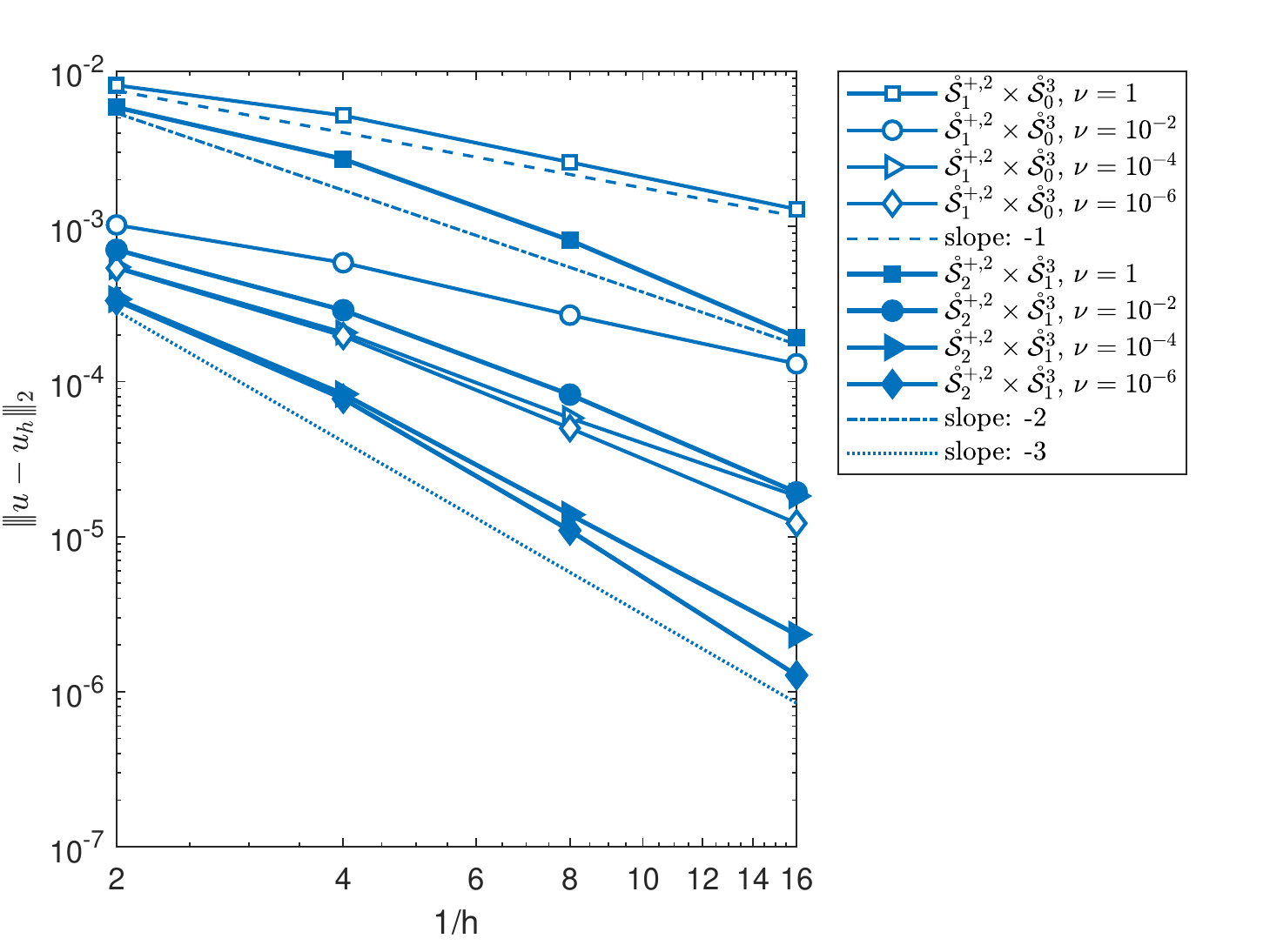}}
\subfloat[The convergence rate of $\|p-p_h\|$.]{\label{converfig:b}\includegraphics[width=0.5\textwidth]{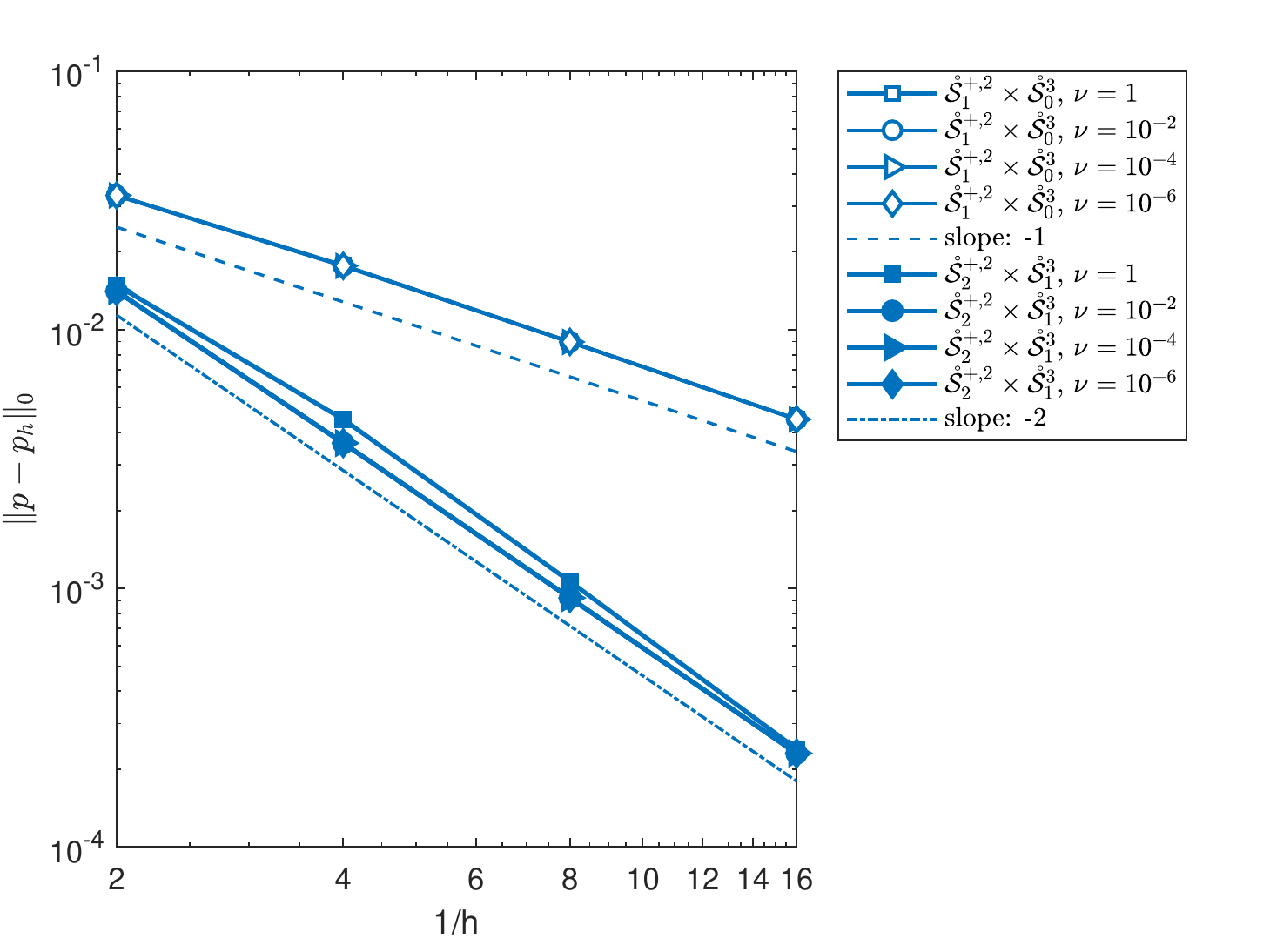}}
\caption{Convergence rates for Example 2 }
\label{fig:rate}
\end{figure}

% \begin{figure}
% \centering
% \begin{minipage}{0.5\textwidth}
% \centering
% \includegraphics[scale=0.42]{umuuh.pdf}
% \caption{$\|\bm{u}-\bm{u}_{h}\|$}
% \label{L2:DC}
% \end{minipage}~~
% \begin{minipage}{0.5\textwidth}
% \centering
% \includegraphics[scale=0.42]{gradumuuh.pdf}
% \caption{$\|\grad_{h}(\bm{u}-\bm{u}_{h})\|$}
% \label{H1:DC}
% \end{minipage}
% \end{figure}

% \begin{figure}
% \centering
% \begin{minipage}{0.5\textwidth}
% \centering
% \includegraphics[scale=0.42]{pmph.pdf}
% \caption{$\|p-p_{h}\|$}
% \label{L2:DC:p}
% \end{minipage}~~
% \begin{minipage}{0.5\textwidth}
% \centering
% \includegraphics[scale=0.42]{quadcurl.pdf}
% \caption{$\|\bm{u}-\bm{u}_{h}\|_{\grad\curl,h}$}
% \label{quadcurl}
% \end{minipage}
% \end{figure}

\section{Conclusion}
In this paper, we constructed a family of $H(\grad\curl)$-nonconforming elements for the $-\curl\Delta\curl$ problems and a family of uniformly stable $[H^1]^3$-nonconforming cubical elements for the Brinkman problem. The two families of elements can fit into a de Rham complex starting with the serendipity finite element space and ending with the piecewise polynomial space. Because of their $H(\curl)$ and $H(\div)$-conformity, they can also be applied to the Darcy problem and Maxwell equations even though this is not a focus of this paper. 

The newly-proposed $H(\grad\curl)$-nonconforming element has at least 48 DOFs on each element. To further reduce the number of DOFs, we will construct fully nonconforming elements in future works.

In this paper, we considered cubical elements, while it is not hard to extend this construction to rectangular meshes. In the future, we will extend the construction to general quadrilateral and hexahedral meshes.

\bibliographystyle{plain}
\bibliography{quadcurl-3d-V4}{}

\begin{thebibliography}{10}

\bibitem{arnold2011serendipity}
D.~N. Arnold and G.~Awanou.
\newblock The serendipity family of finite elements.
\newblock {\em Foundations of computational mathematics}, 11(3):337--344, 2011.

\bibitem{arnold2014finite}
D.~N. Arnold and G.~Awanou.
\newblock Finite element differential forms on cubical meshes.
\newblock {\em Mathematics of Computation}, 83(288):1551--1570, 2014.

\bibitem{arnold1984stable}
D.~N. Arnold, F.~Brezzi, and M.~Fortin.
\newblock A stable finite element for the {Stokes} equations.
\newblock {\em Calcolo}, 21(4):337--344, 1984.

\bibitem{arnold1992quadratic}
D.~N Arnold and J.~Qin.
\newblock Quadratic velocity/linear pressure {Stokes elements}.
\newblock {\em Advances in computer methods for partial differential
  equations}, 7:28--34, 1992.

\bibitem{brenner2003poincare}
S.~C. Brenner.
\newblock Poincar{\'e}--{F}riedrichs inequalities for piecewise ${H}^1$
  functions.
\newblock {\em SIAM Journal on Numerical Analysis}, 41(1):306--324, 2003.

\bibitem{chacon2007steady}
L.~Chac{\'o}n, A.~N. Simakov, and A.~Zocco.
\newblock {Steady-state properties of driven magnetic reconnection in 2D
  electron magnetohydrodynamics}.
\newblock {\em Physical review letters}, 99(23):235001, 2007.

\bibitem{chen2017uniformly}
S.~Chen, L.~Dong, and J.~Zhao.
\newblock Uniformly convergent cubic nonconforming element for {Darcy--Stokes}
  problem.
\newblock {\em Journal of Scientific Computing}, 72(1):231--251, 2017.

\bibitem{christiansen2018generalized}
S.~H. Christiansen and K.~Hu.
\newblock Generalized finite element systems for smooth differential forms and
  {Stokes'} problem.
\newblock {\em Numerische Mathematik}, 140(2):327--371, 2018.

\bibitem{crouzeix1973conforming}
M.~Crouzeix and P.-A. Raviart.
\newblock Conforming and nonconforming finite element methods for solving the
  stationary {Stokes} equations {I}.
\newblock {\em Revue fran{\c{c}}aise d'automatique informatique recherche
  op{\'e}rationnelle. Math{\'e}matique}, 7(R3):33--75, 1973.

\bibitem{fu2020exact}
G.~Fu, J.~Guzm{\'a}n, and M.~Neilan.
\newblock {Exact smooth piecewise polynomial sequences on Alfeld splits}.
\newblock {\em Mathematics of Computation}, 89(323):1059--1091, 2020.

\bibitem{gillette2019trimmed}
A.~Gillette and T.~Kloefkorn.
\newblock Trimmed serendipity finite element differential forms.
\newblock {\em Mathematics of Computation}, 88(316):583--606, 2019.

\bibitem{Girault2012Finite}
V.~Girault and P.~Raviart.
\newblock {\em Finite element methods for {N}avier-{S}tokes equations: theory
  and algorithms}, volume~5.
\newblock Springer Science \& Business Media, 2012.

\bibitem{johnny2012family}
J.~Guzm{\'a}n and M.~Neilan.
\newblock A family of nonconforming elements for the brinkman problem.
\newblock {\em IMA Journal of Numerical Analysis}, 32(4):1484--1508, 2012.

\bibitem{guzman2018inf}
J.~Guzm{\'a}n and M.~Neilan.
\newblock Inf-sup stable finite elements on barycentric refinements producing
  divergence--free approximations in arbitrary dimensions.
\newblock {\em SIAM Journal on Numerical Analysis}, 56(5):2826--2844, 2018.

\bibitem{guzman2019scott}
J.~Guzm{\'a}n and L.~Scott.
\newblock The {Scott-Vogelius} finite elements revisited.
\newblock {\em Mathematics of Computation}, 88(316):515--529, 2019.

\bibitem{guzman2020exact}
Johnny Guzm{\'a}n, Anna Lischke, and Michael Neilan.
\newblock Exact sequences on worsey--farin splits.
\newblock {\em Mathematics of Computation}, 91(338):2571--2608, 2022.

\bibitem{hiptmair2002finite}
R.~Hiptmair.
\newblock Finite elements in computational electromagnetism.
\newblock {\em Acta Numerica}, 11:237--339, 2002.

\bibitem{Hu2020Simple}
K.~Hu, Q.~Zhang, and Z.~Zhang.
\newblock Simple curl-curl-conforming finite elements in two dimensions.
\newblock {\em SIAM Journal on Scientific Computing}, 42(6):A3859--A3877, 2020.

\bibitem{Hu2020Simple2}
K.~Hu, Q.~Zhang, and Z.~Zhang.
\newblock Simple curl-curl-conforming finite elements in two dimensions
  (erratum).
\newblock {\em arXiv:2004.12507v2}, 2020.

\bibitem{Hu2022A}
K.~Hu, Q.~Zhang, and Z.~Zhang.
\newblock A family of finite element stokes complexes in three dimensions.
\newblock {\em SIAM Journal on Numerical Analysis}, 60(1):222--243, 2022.

\bibitem{huang2020nonconforming}
Xuehai Huang.
\newblock Nonconforming finite element stokes complexes in three dimensions.
\newblock {\em Science China Mathematics}, pages 1--24, 2023.

\bibitem{mardal2002robust}
K.~A. Mardal, X.~Tai, and R.~Winther.
\newblock A robust finite element method for {Darcy--Stokes flow}.
\newblock {\em SIAM Journal on Numerical Analysis}, 40(5):1605--1631, 2002.

\bibitem{mindlin1962effects}
R.~D. Mindlin and H.~F. Tiersten.
\newblock Effects of couple-stresses in linear elasticity.
\newblock {\em Archive for Rational Mechanics and Analysis}, 11(1):415--448,
  1962.

\bibitem{Monk2003}
P.~Monk.
\newblock {\em Finite Element Methods for {M}axwell's Equations}.
\newblock Oxford University Press, 2003.

\bibitem{neilan2015discrete}
M.~Neilan.
\newblock {Discrete and conforming smooth de Rham complexes in three
  dimensions}.
\newblock {\em Mathematics of Computation}, 84(295):2059--2081, 2015.

\bibitem{neilan2016stokes}
M.~Neilan and D.~Sap.
\newblock Stokes elements on cubic meshes yielding divergence-free
  approximations.
\newblock {\em Calcolo}, 53(3):263--283, 2016.

\bibitem{park2008variational}
S.~K. Park and X.~Gao.
\newblock Variational formulation of a modified couple stress theory and its
  application to a simple shear problem.
\newblock {\em Zeitschrift f{\"u}r angewandte Mathematik und Physik},
  59(5):904--917, 2008.

\bibitem{tai2006discrete}
X.~Tai and R.~Winther.
\newblock A discrete de {R}ham complex with enhanced smoothness.
\newblock {\em Calcolo}, 43(4):287--306, 2006.

\bibitem{taylor1973numerical}
C.~Taylor and P.~Hood.
\newblock A numerical solution of the {Navier-Stokes} equations using the
  finite element technique.
\newblock {\em Computers \& Fluids}, 1(1):73--100, 1973.

\bibitem{wang2021hh}
L.~Wang, H.~Li, and Z.~Zhang.
\newblock {$H(\text{curl}^2$)}-conforming spectral element method for quad-curl
  problems.
\newblock {\em Computational Methods in Applied Mathematics}, 21(3):661--681,
  2021.

\bibitem{wang2021h}
L.~Wang, W.~Shan, H.~Li, and Z.~Zhang.
\newblock {$H(\text{curl}^2$)-conforming quadrilateral spectral element method
  for quad-curl problems}.
\newblock {\em Mathematical Models and Methods in Applied Sciences},
  31(10):1951--1986, 2021.

\bibitem{xie2008uniformly}
X.~Xie, J.~Xu, and G.~Xue.
\newblock Uniformly-stable finite element methods for {Darcy-Stokes-Brinkman}
  models.
\newblock {\em Journal of Computational Mathematics}, pages 437--455, 2008.

\bibitem{xu2010new}
X.~Xu and S.~Zhang.
\newblock A new divergence-free interpolation operator with applications to the
  {Darcy--Stokes--Brinkman equations}.
\newblock {\em SIAM Journal on Scientific Computing}, 32(2):855--874, 2010.

\bibitem{WZZelement}
Q.~Zhang, L.~Wang, and Z.~Zhang.
\newblock {$H(\text{curl}^2$)}-conforming finite elements in 2 dimensions and
  applications to the quad-curl problem.
\newblock {\em SIAM Journal on Scientific Computing}, 41(3):A1527--A1547, 2019.

\bibitem{ZhangCSIAM2021A}
Q.~Zhang and Z.~Zhang.
\newblock A family of curl-curl conforming finite elements on tetrahedral
  meshes.
\newblock {\em CSIAM Transactions on Applied Mathematics}, 1(4):639--663, 2020.

\bibitem{zhang2005new}
S.~Zhang.
\newblock {A new family of stable mixed finite elements for the 3D Stokes
  equations}.
\newblock {\em Mathematics of computation}, 74(250):543--554, 2005.

\bibitem{zhang2008p1}
S.~Zhang.
\newblock {On the $P_1$ Powell-Sabin divergence-free finite element for the
  Stokes equations}.
\newblock {\em Journal of Computational Mathematics}, pages 456--470, 2008.

\bibitem{zhang2009family}
S.~Zhang.
\newblock A family of {$Q_{k+1,k}\times Q_{k,k+1}$ }divergence-free finite
  elements on rectangular grids.
\newblock {\em SIAM journal on numerical analysis}, 47(3):2090--2107, 2009.

\bibitem{zhang2011quadratic}
S.~Zhang.
\newblock {Quadratic divergence-free finite elements on Powell--Sabin
  tetrahedral grids}.
\newblock {\em Calcolo}, 48(3):211--244, 2011.

\bibitem{Zheng2011A}
B.~Zheng, Q.~Hu, and J.~Xu.
\newblock A nonconforming finite element method for fourth order curl equations
  in $\mathbb{R}^3$.
\newblock {\em Mathematics of Computation}, 80(276):1871--1886, 2011.

\bibitem{zhou2021low}
X.~Zhou, Z.~Meng, and J.~Su.
\newblock Low-order nonconforming brick elements for the {3D Brinkman} model.
\newblock {\em Computers \& Mathematics with Applications}, 98:201--217, 2021.

\end{thebibliography}
~
\\

\end{document}